\documentclass[10pt, a4paper]{amsart}
\usepackage{amsmath, amssymb,amsthm,amsfonts,amsmath,euscript,mathtools,relsize}
\usepackage[extra]{tipa}
\usepackage{lettrine}
\usepackage[Sonny]{fncychap}
\usepackage[english]{babel}
\usepackage{amsmath, amssymb,amsthm}
\usepackage[all]{xy}
\usepackage{tabularx}
\usepackage[applemac]{inputenc}
\usepackage{graphicx}
\usepackage[dvipsnames]{xcolor}
\usepackage{hyperref}
\usepackage[applemac]{inputenc}
\usepackage{enumitem}
\usepackage{geometry}
\usepackage{prerex}
\usetikzlibrary{fit}
\usepackage{pdflscape}
\usepackage{tikz-cd} 
\usepackage{mathtools}

\makeatletter
\def\oversortoftilde#1{\mathop{\vbox{\m@th\ialign{##\crcr\noalign{\kern3\p@}%
      \sortoftildefill\crcr\noalign{\kern3\p@\nointerlineskip}%
      $\hfil\displaystyle{#1}\hfil$\crcr}}}\limits}

\def\sortoftildefill{$\m@th \setbox\z@\hbox{$\braceld$}%
  \braceld\leaders\vrule \@height\ht\z@ \@depth\z@\hfill\braceru$}

\makeatother

\usepackage{mathrsfs}

\DeclareMathOperator{\B}{\overline{B}}

\DeclareMathOperator{\Ker}{Ker}

\DeclareMathOperator{\ord}{ord}
\DeclareMathOperator{\Gal}{Gal}

\DeclareMathOperator{\SL}{SL}

\DeclareMathOperator{\PSL}{PSL}

\newcommand{\cf}{\textit{cf. }}
\newcommand{\ie}{\textit{i.e. }}

\newcommand{\eg}{\textit{e.g. }}
\theoremstyle{definition}

\newtheorem*{notation}{Notation}

\newtheorem{rem}{Remark}

\theoremstyle{plain}
\newtheorem{thm}{Theorem}[section]
\newtheorem{defn}[thm]{Definition}
\newtheorem{conj}[thm]{Conjecture}
\newtheorem{lem}[thm]{Lemma}
\newtheorem{corr}[thm]{Corollary}

\newtheorem{prop}[thm]{Proposition}
\newcommand{\Z}{\mathbf Z}
\newcommand{\Q}{\mathbf Q}

\newcommand{\C}{\mathbf C}

\newcommand{\Zp}{\mathbf Z_{p}}
\newcommand{\Qp}{\mathbf Q_{p}}
\newcommand{\Fp}{\mathbf{F}_{p}}
\renewcommand{\O}{\mathcal O}
\newcommand{\T}{\mathbb{T}}

\def\makeop#1{\expandafter\def\csname#1\endcsname
  {\mathop{\rm #1}\nolimits}\ignorespaces}
\makeop{Hom}   \makeop{End}   \makeop{Aut}   \makeop{Isom}  \makeop{Pic}
\makeop{ord}   \makeop{Char}  \makeop{Div}   \makeop{Lie}

\title{On a conjecture of Sharifi and Mazur's Eisenstein ideal}
\author{Emmanuel Lecouturier and Jun Wang}

\begin{document}
\maketitle

\begin{abstract}
Let $N$ and $p$ be prime numbers $\geq 5$ such that $p$ divides $N-1$. Let $I$ be Mazur's Eisenstein ideal of level $N$ and $H_+$ be the plus part of $H_1(X_0(N), \Zp)$ for the complex conjugation. We give a conjectural explicit description of the group $I\cdot H_+/I^2\cdot H_+$ in terms of the second $K$-group of the cyclotomic field $\Q(\zeta_N)$. We prove that this conjecture follows from a conjecture of Sharifi about some Eisenstein ideal of level $\Gamma_1(N)$. Following the work of Fukaya--Kato, we prove partial results on Sharifi's conjecture. This allows us to prove partial results on our conjecture.
\end{abstract}
\tableofcontents

\section{Introduction}\label{introduction}
Let $N$ and $p$ be prime numbers $\geq 5$ such that $p$ divides $N-1$. Let $\T=\Zp[T_n, n\geq 1]$ be the Hecke algebra acting on $S_2(\Gamma_0(N), \Zp)$, the space of cusp forms of weight $2$ and level $\Gamma_0(N)$ over $\Zp$. Let $I \subset \T$ be Mazur's Eisenstein ideal \cite{Mazur_Eisenstein}, \ie the ideal generated by the Hecke operators $T_{\ell}-\ell-1$ for primes $\ell \neq N$ and by $T_N-1$. We denote by $H$ (resp. $H_+$) the singular homology group $H_1(X_0(N), \Zp)$ (resp. the subgroup fixed by the complex conjugation in $H_1(X_0(N), \Zp)$). 

Barry Mazur constructed an explicit isomorphism $H_+/I\cdot H_+ \xrightarrow{\sim} (\Z/N\Z)^{\times} \otimes_{\Z} \Zp$. This fundamental construction is used in numerous applications, \eg \cite{Merel_accouplement}, \cite{Merel_cyclotomic}, \cite{Lecouturier_MT}, \cite{Lecouturier_higher_Eisenstein}. This also allowed Mazur to give a criterion for the Hecke operator $T_{\ell}-\ell-1$ to locally generate $I$ \cite[Theorem 18.10]{Mazur_Eisenstein}.

The results of Mazur imply that the group $I\cdot H_+/I^2\cdot H_+$ is canonically isomorphic to $\left((\Z/N\Z)^{\times} \otimes_{\Z} \Zp\right)^{\otimes 2}$. Lo\"ic Merel asked for an explicit description of this isomorphism \cite[p. 104 ``\`A quoi est \'egal...?'']{Merel_cyclotomic}. Our paper can be considered as an attempt to answer his question. We now give more details about our work.

Let $\xi : \Zp[(\Z/N\Z)^{\times}] \rightarrow H_+$ be the  map given by $$\xi([a])=\frac{1}{2}\cdot(\{\infty, \frac{a}{N}\}+\{\infty, \frac{-a}{N}\})$$
where if $\alpha$, $\beta$ $\in \mathbf{P}^1(\Q)$, we denote by $\{\alpha, \beta\}$ the homology class of the geodesic path connecting $\alpha$ and $\beta$ in the upper-half plane (a similar notation applies to other modular curves). This is the plus part of the Manin map \cite[\S 1.5]{Manin} composed with the Atkin--Lehner involution (we make this modification to be consistent with \cite{Fukaya_Kato}). The map $\xi$ is well-known to be surjective (\cf \cite[Proposition 3]{Merel_accouplement}).

\begin{thm}[Mazur]\cite[Proposition II.18.8]{Mazur_Eisenstein}
There is a unique group isomorphism
$$H_+/I\cdot H_+ \xrightarrow{\sim} (\Z/N\Z)^{\times} \otimes_{\Z} \Zp$$
sending $\xi([a])$ to $a \otimes 1$, for all $a \in (\Z/N\Z)^{\times}$.
\end{thm}

In other words, we have:
\begin{equation}\label{Mazur_I/I^2}
I\cdot H_+ = \{\sum_{a \in (\Z/N\Z)^{\times}} \lambda_a\cdot \xi([a]), \lambda_a \in \Zp,  \text{ such that } \sum_{a \in (\Z/N\Z)^{\times}} a \otimes \lambda_a = 0 \in (\Z/N\Z)^{\times} \otimes_{\Z} \Zp\} \text{ .}
\end{equation}

To describe $I^2\cdot H_+$, we need to introduce some more notation. If $A$ is a ring, we denote by $K_2(A)$ the second $K$-group as defined by Quillen. We let $\mathcal{K} = K_2(\Z[\zeta_N, \frac{1}{Np}]) \otimes_{\Z} \Zp$, where $\zeta_N$ is a primitive $N$th root of unity in an algebraic closure of $\Q$. There is an action of the group ring $\Lambda:= \Zp[\Gal(\Q(\zeta_N)/\Q)]$ on $\mathcal{K}$. 
In this paper, we choose the ``inverse'' of the natural action. Thus, we have $$\sigma_a\cdot \{1-\zeta_N^u, 1-\zeta_N^v\} =  \{1-\zeta_N^{a^{-1}u}, 1-\zeta_N^{a^{-1}v}\}$$ where if $a\in (\Z/N\Z)^{\times}$ then $\sigma_a \in \Gal(\Q(\zeta_N)/\Q)$ is characterized by $\sigma_a(\zeta_N)=\zeta_N^a$.

We denote by $J$ the augmentation ideal of $\Lambda$. If $x$, $y$ $\in \Z[\zeta_N, \frac{1}{Np}]^{\times}$, let $\langle x,y\rangle$ be the associated Steinberg symbol in $\mathcal{K}$.

\begin{conj}\label{conjecture_I/I^2}
There is a group isomorphism
$$I\cdot H_+/I^2\cdot H_+ \xrightarrow{\sim} J\cdot \mathcal{K}/J^2\cdot \mathcal{K}$$
sending $\sum_{a \in (\Z/N\Z)^{\times}} \lambda_a\cdot \xi([a])$ to $$\sum_{a \in (\Z/N\Z)^{\times}} \lambda_a\cdot (\langle 1-\zeta_N^a, 1-\zeta_N\rangle-\frac{1}{2}\cdot ([\sigma_a]-1)\cdot \langle 1-\zeta_N^a, 1-\zeta_N\rangle) \text{ .}$$

\end{conj}

\begin{rem}\label{remark_augmentation_K}
By Proposition \ref{Prop_K_theory} (\ref{Prop_K_theory_augmentation}), the condition 
$$\sum_{a \in (\Z/N\Z)^{\times}} a \otimes \lambda_a = 0 \in (\Z/N\Z)^{\times} \otimes_{\Z} \Zp$$
is equivalent to
$$\sum_{a \in (\Z/N\Z)^{\times}} \lambda_a\cdot \langle1-\zeta_N^a, 1-\zeta_N\rangle \in J \cdot \mathcal{K} \text{ .}$$
\end{rem}

We now relate Conjecture \ref{conjecture_I/I^2} to a conjecture of Romyar Sharifi. Let $X_1(N)$ be the compact modular curve of level $\Gamma_1(N)$. Let $C_{\infty}$ be the set of cusps of $X_1(N)$ above the cusp $\Gamma_0(N)\infty$ of $X_0(N)$. We denote by $H_1(X_1(N), C_{\infty}, \Zp)$ the singular homology group of $X_1(N)$ relative to $C_{\infty}$ (with coefficients in $\Zp$). Let $H_1(X_1(N), C_{\infty}, \Zp)_+$ be the subgroup of $H_1(X_1(N), C_{\infty}, \Zp)$ fixed by the complex conjugation.

There is a surjective map (\cf Proposition \ref{generation_Manin_C_0^{(p^r)}}) 
$$\xi_1 : \Z[((\Z/N\Z)^{\times})^2]\rightarrow H_1(X_1(N), C_{\infty}, \Zp)_+$$
given by $$\xi_1([(u,v)]) = \frac{1}{2}\cdot(\{\frac{-d}{bN}, \frac{-c}{aN}\}+\{\frac{d}{bN}, \frac{c}{aN}\})$$
where $\begin{pmatrix}a & b \\ c & d \end{pmatrix} \in \SL_2(\Z)$ is such that $c \equiv u \text{ (modulo }N\text{)}$ and $d \equiv v \text{ (modulo }N\text{)}$. 

One can prove (\cf Proposition \ref{Prop_existence_varpi}) that the map $\Z[((\Z/N\Z)^{\times})^2]\rightarrow \mathcal{K}$ sending $[(u,v)]$ to $\langle 1-\zeta_N^u, 1-\zeta_N^v\rangle$ factors through $\xi_1$, thus giving a map
$$\tilde{\varpi} : H_1(X_1(N), C_{\infty}, \Zp)_+ \rightarrow \mathcal{K} \text{ .}$$

\begin{conj}[Sharifi]\label{Sharifi_conjecture}
The map $\tilde{\varpi}$ is annihilated by the Hecke operators $T_{\ell}-\ell\cdot \langle \ell \rangle -1$ for primes $\ell \neq N$ and by $T_{N}-1$ (here the Hecke operators are the usual ones, induced by Albanese functoriality).
\end{conj}

\begin{rem}\label{Remark_Sharifi_Venkatesh}
Sharifi announced in a talk at the workshop ``Eisenstein ideal and Iwasawa theory'' held in Beijing in June 2019 that, together with Akshay Venkatesh, he was able to prove that the restriction of $\tilde{\varpi}$ to $H_1(X_1(N), \Zp)$ is annihilated by the Hecke operators $T_{\ell}-\ell\cdot \langle \ell \rangle -1$ for primes $\ell \neq N$. 
\end{rem}

\begin{rem}\label{Remark_identity_Buisioc}
Following the techniques of Cecilia Busuioc and Sharifi, one can prove that $\tilde{\varpi}$ is annihilated by the Hecke operators $T_{\ell}-\ell\cdot \langle \ell \rangle -1$ for $\ell \in \{2,3\}$ \cite[Theorem 1.2]{Busuioc}. It seems however that such an approach does not work for $\ell \geq 5$. For $\ell=5$, we need to prove that for any $u$, $v$ $\in (\Z/N\Z)^{\times}$, we have in $\mathcal{K}$:
\begin{align*}
&\langle 1-\zeta_N^{5u}, 1-\zeta_N^{v} \rangle + \langle 1-\zeta_N^{5u}, 1-\zeta_N^{2u+v} \rangle - \langle 1-\zeta_N^{u-2v}, 1-\zeta_N^{2u+v} \rangle+\langle 1-\zeta_N^{u-2v}, 1-\zeta_N^{5v} \rangle \\& + \langle 1-\zeta_N^{5u}, 1-\zeta_N^{4u+v} \rangle - \langle 1-\zeta_N^{-3u-2v}, 1-\zeta_N^{4u+v} \rangle +\langle 1-\zeta_N^{3u+2v}, 1-\zeta_N^{2u+3v} \rangle  \\& -\langle 1-\zeta_N^{-u-4v}, 1-\zeta_N^{2u+3v} \rangle +\langle 1-\zeta_N^{u+4v}, 1-\zeta_N^{5v} \rangle  + \langle 1-\zeta_N^{5u}, 1-\zeta_N^{-2u+v} \rangle \\& - \langle 1-\zeta_N^{u+2v}, 1-\zeta_N^{-2u+v} \rangle   + \langle 1-\zeta_N^{u+2v}, 1-\zeta_N^{5v} \rangle  + \langle 1-\zeta_N^{5u}, 1-\zeta_N^{-4u+v} \rangle \\& - \langle 1-\zeta_N^{-3u+2v}, 1-\zeta_N^{-4u+v} \rangle + \langle 1-\zeta_N^{3u-2v}, 1-\zeta_N^{-2u+3v} \rangle  - \langle 1-\zeta_N^{-u+4v}, 1-\zeta_N^{-2u+3v} \rangle  \\& + \langle 1-\zeta_N^{u-4v}, 1-\zeta_N^{5v} \rangle + \langle 1-\zeta_N^{u}, 1-\zeta_N^{5v} \rangle\\& =\langle 1-\zeta_N^{5u}, 1-\zeta_N^{5v}\rangle + 5\cdot \langle 1-\zeta_N^{u}, 1-\zeta_N^{v} \rangle \text{ .}
\end{align*}
We were not able to prove this identity.
\end{rem}

A study of the kernel and image of the natural map $H_1(X_1(N), C_{\infty}, \Zp)_+ \rightarrow H_+$ leads to the following result.

\begin{thm}\label{thm_conj}
Conjecture \ref{Sharifi_conjecture} implies Conjecture \ref{conjecture_I/I^2}.
\end{thm}

\begin{rem}\label{Remark_Sharifi_Venkatesh}
If we only assume that the restriction of $\tilde{\varpi}$ to $H_1(X_1(N), \Zp)$ is annihilated by the Hecke operators $T_{\ell}-\ell\cdot \langle \ell \rangle -1$ for primes $\ell \neq N$ and by $T_N-1$ (\cf Remark \ref{Remark_Sharifi_Venkatesh}), the proof of Theorem \ref{thm_conj} gives the existence of the map $I\cdot H_+/I^2\cdot H_+ \rightarrow J\cdot \mathcal{K}/J^2\cdot \mathcal{K}$ of Conjecture \ref{conjecture_I/I^2}. However, we do not know how to prove that this map is surjective without assuming Conjecture \ref{Sharifi_conjecture}.
\end{rem}

The embedding $\Z[\frac{1}{Np}] \hookrightarrow \Z[\zeta_N, \frac{1}{Np}]$ yields a map on $K$-groups $K_2(\Z[\frac{1}{Np}]) \rightarrow K_2(\Z[\zeta_N, \frac{1}{Np}])$. One can show that we have a canonical group isomorphism $K_2(\Z[\frac{1}{Np}]) \otimes_{\Z} \Zp \simeq (\Z/N\Z)^{\times} \otimes_{\Z} \Zp$ and that the map $K_2(\Z[\frac{1}{Np}]) \otimes_{\Z} \Zp \rightarrow \mathcal{K}$ has image $\nu \cdot \mathcal{K}$ where $\nu = \sum_{g \in \Gal(\Q(\zeta_N)/\Q)} [g]$ is the norm. We prove that Conjecture (\ref{Sharifi_conjecture}) is true after quotienting $\tilde{\varpi}$ by $\nu \cdot \mathcal{K}$.

\begin{thm}\label{thm_Sharifi_mod_norms}
The map $H_1(X_1(N), C_{\infty}, \Zp)_+ \rightarrow \mathcal{K}/\nu\cdot \mathcal{K}$ obtained from $\tilde{\varpi}$ is annihilated by the Hecke operators $T_{\ell}-\ell\cdot \langle \ell \rangle -1$ for primes $\ell \neq N$ and by $T_{N}-1$.
\end{thm}

\begin{rem}
This proves that the identity of Remark \ref{Remark_identity_Buisioc} holds in $\mathcal{K}/\nu\cdot \mathcal{K}$. However, we could not find a direct or elementary proof. 
\end{rem}

Our proof is inspired by -- and rely on -- the unpublished work of Fukaya Takako and Kazuya Kato on a conjecture of Sharifi similar to Conjecture \ref{Sharifi_conjecture} \cite{Fukaya_Kato}. In particular, we make use of Beilinson--Kato elements in the $K_2$ group of the modular curve $Y_1(N)$. Fukaya and Kato assume that the level $N$ is divisible by $p$. The analogue of $\nu\cdot \mathcal{K}$ in their situation is the group $H^2(\Z[\tfrac{1}{Np}], H^{0, \text{ord}}_{\textnormal{\'et}}(Y_1(N)(2)))$, where $\text{ord}$ means the ordinary part. The group $H^{0, \text{ord}}_{\textnormal{\'et}}(Y_1(N)(2))$ is trivial in their case, since $T_p$ acts on  by multiplication by $p$ on $H^{0}_{\textnormal{\'et}}(Y_1(N)(2))=\Zp$ (\cf \cite[Lemma 5.2.5]{Fukaya_Kato}). In our situation there are no ordinariness considerations and the group $\nu\cdot \mathcal{K} \simeq H^2(\Z[\tfrac{1}{Np}], H^{0}_{\textnormal{\'et}}(Y_1(N)(2))) = H^2(\Z[\tfrac{1}{Np}], \Zp(2))$ is non-zero (and is in fact canonically isomorphic to $(\Z/N\Z)^{\times} \otimes_{\Z} \Zp$). It thus seems that the method of Fukaya--Kato alone is not enough to go beyond Theorem \ref{thm_Sharifi_mod_norms}.

Using Theorem \ref{thm_Sharifi_mod_norms}, a result of Merel \cite{Merel_accouplement} and a result of Cornelius Greither and Christian D. Popescu \cite{GP}, we are able to prove:

\begin{thm}\label{thm_conj_2}
Conjecture (\ref{conjecture_I/I^2}) holds if the image of the integer $\prod_{k=1}^{\frac{N-1}{2}}k^k$ in $(\Z/N\Z)^{\times} \otimes_{\Z} \Zp$ is trivial.
\end{thm}

\begin{rem}
The condition that the integer $\prod_{k=1}^{\frac{N-1}{2}}k^k$ in $(\Z/N\Z)^{\times} \otimes_{\Z} \Zp$ is trivial is related to the Newton polygon of the completion of $\mathbb{T}$ at $I$ (\cf \cite[Corollary 1.4.2 and Proposotion 1.5.3]{WWE}).
\end{rem}

There is a ``modulo $p^s$'' version of Conjecture (\ref{conjecture_I/I^2}) that we can prove in all cases using Theorem \ref{thm_Sharifi_mod_norms}. Let $t = \ord_p(N-1)$ be the $p$-adic valuation of $N-1$ and $s$ be an integer such that $1 \leq s \leq t$. 

\begin{thm}\label{I/I^2_modulo_p}
There is a group isomorphism
$$I\cdot (H_+ \otimes_{\Zp} \Z/p^s\Z) /I^2\cdot (H_+ \otimes_{\Zp} \Z/p^s\Z) \xrightarrow{\sim} J\cdot (\mathcal{K}\otimes_{\Zp} \Z/p^s\Z)/J^2\cdot (\mathcal{K}\otimes_{\Zp} \Z/p^s\Z)$$
sending $\sum_{a \in (\Z/N\Z)^{\times}} \lambda_a\cdot \xi([a])$ to $$\sum_{a \in (\Z/N\Z)^{\times}} \lambda_a\cdot (\langle 1-\zeta_N^a, 1-\zeta_N\rangle-\frac{1}{2}\cdot ([\sigma_a]-1)\cdot \langle 1-\zeta_N^a, 1-\zeta_N\rangle) \text{ ,}$$
where $\lambda_a \in \Z/p^s\Z$.
\end{thm}

\begin{rem}
Theorem \ref{I/I^2_modulo_p} is equivalent to the determination of the second \textit{higher Eisenstein element} in $H_1(Y_0(N), \Z/p^s\Z)^-$ in the language of \cite{Lecouturier_higher_Eisenstein}.
\end{rem}

\subsection*{Acknowledgements}
This paper grew out of the authors PhD theses. We want to thank our respective advisors Lo\"ic Merel and Romyar Sharifi for their support and advice during the completion of this work. The first author was funded by Universit\'e Paris--Diderot during his PhD thesis and by the Yau Mathematical Sciences Center and Tsinghua university during the writing of this paper. The second author was funded by the University of Arizona during his PhD thesis and partially supported by National Science Foundation under Grants No. DMS-1360583 and No. DMS-1401122. The second author would like to thank Sujatha Ramdorai for supporting his
postdoctoral studies. Finally, we would like to thank the Yau Mathematical Sciences Center for inviting the second author on December 2018, when this project began.

\section{Refined Hida theory}\label{odd_modSymb_section_hida}
In this section, we will use the following notation.
 \begin{itemize}
  \item $\sigma = \begin{pmatrix}
0 & -1 \\
1 & 0
\end{pmatrix} $ and $\tau = \begin{pmatrix}
0 & -1 \\
1 & -1
\end{pmatrix}$ $\in \SL_2(\Z)$.
 \item If $\Gamma$ is a subgroup of $\Gamma_0(N)$ containing $\Gamma_1(N)$, let $X_{\Gamma}$ be the compact modular curve associated to $\Gamma$. 
 \item  $C_{\Gamma}^{\infty}$ (resp. $C_{\Gamma}^{0}$) is the set of cusps of $X_{\Gamma}$ above the cusp $\Gamma_0(N) \cdot \infty$ (resp. $\Gamma_0(N) \cdot 0$) of $X_0(N)$. 
 \item $C_{\Gamma} = C_{\Gamma}^0 \cup C_{\Gamma}^{\infty}$
 \item  $\tilde{H}_{\Gamma}' = H_1(X_{\Gamma}, C_{\Gamma}, \Zp)$, $\tilde{H}_{\Gamma} = H_1(X_{\Gamma}, C_{\Gamma}^{\infty}, \Zp)$ and $H_{\Gamma} = H_1(X_{\Gamma}, \Zp)$. 
 \item $\partial : \tilde{H}_{\Gamma}'  \rightarrow \Zp[C_{\Gamma}]^0$ is the boundary map, sending the geodesic path $\{\alpha, \beta\}$ to $[\beta]-[\alpha]$ where $\alpha$, $\beta$ $\in \mathbf{P}^1(\mathbf{Q})$ and $\Zp[C_{\Gamma}]^0$ is the augmentation subgroup of $\Zp[C_{\Gamma}]$.
 \item $\left(\tilde{H}_{\Gamma}\right)_+$ is the subgroup of elements of $\tilde{H}_{\Gamma}$ fixed by the complex conjugation. A similar notation applies to $H_{\Gamma}$.
 \item $D_{\Gamma} \subset (\Z/N\Z)^{\times}$ is the subgroup generated by the classes of the lower right corners of the elements of $\Gamma$ and by the class of $-1$. 
 \item $\Lambda_{\Gamma} = \Zp[(\Z/N\Z)^{\times}/D_{\Gamma}]$.
 \item If $\Gamma_1$ and $\Gamma_2$ are subgroups of $\SL_2(\Z)$ such that $\Gamma_1(N) \subset \Gamma_1 \subset \Gamma_2 \subset \Gamma_0(N)$, we let $J_{\Gamma_1\rightarrow \Gamma_2} = \Ker(\Lambda_{\Gamma_1} \rightarrow \Lambda_{\Gamma_2})$. It is a principal ideal of $\Lambda_{\Gamma_1}$, generated by $[x]-1$ where $x$ is a generator of $\Ker((\Z/N\Z)^{\times}/D_{\Gamma_1} \rightarrow (\Z/N\Z)^{\times}/D_{\Gamma_2})$.  
 \item $\tilde{\mathbb{T}}_{\Gamma}'$ (resp. $\tilde{\mathbb{T}}_{\Gamma}$, resp. $\mathbb{T}_{\Gamma}$) is the $\Zp$-Hecke algebra acting faithfully on $\tilde{H}_{\Gamma}'$ (resp. $\tilde{H}_{\Gamma}$, resp. $H_{\Gamma}$) generated by the Hecke operators $T_n$ for $n \geq 1$ and the diamond operators (induced by the Albanese functoriality). The $d$th diamond operator is denoted by $\langle d \rangle$. By convention, it corresponds on modular form to the action of a matrix whose lower right corner is congruent to $d$ modulo $N$.
\end{itemize}

We will need some ``refined Hida control'' results, describing the kernel of the various maps in homology induced by the degeneracy maps between the various modular curves. 

Manin proved \cite[Theorem 1.9]{Manin} that we have a surjection
$$\mathcal{\xi}_{\Gamma} : \Zp[\Gamma \backslash \PSL_2(\Z)] \rightarrow H_1(X_{\Gamma}, C_{\Gamma}, \Zp)$$
such that $\xi_{\Gamma}(\Gamma\cdot g)$ is the class of the geodesic path $w_N\{g(0), g(\infty)\}$, $w_N$ being the Atkin--Lehner involution (induced by the map $z \mapsto -\frac{1}{Nz}$ in the upper-half plane). Furthermore, he proved that the kernel of $\xi_{\Gamma}$ is spanned by the sum of the (right) $\sigma$-invariants and $\tau$-invariants.

Recall that $\Gamma_1(N) \subset \Gamma \subset \Gamma_0(N)$. Consider the bijection $$\kappa : \Gamma \backslash \PSL_2(\Z) \xrightarrow{\sim} \left((\Z/N\Z)^2\backslash \{(0,0)\} \right) / D_{\Gamma}$$ given by $\kappa(\Gamma \cdot g)= [c,d]$
where $g = \begin{pmatrix}
a & b \\
c & d
\end{pmatrix}$ and $[c,d]$ is the class of $(c,d)$ modulo $D_{\Gamma}$. By abuse of notation, we identify $\Gamma \backslash \PSL_2(\Z)$ and $\left((\Z/N\Z)^2\backslash \{(0,0)\} \right) / D_{\Gamma}$.

The map $(\Z/N\Z)^{\times}/ D_{\Gamma} \rightarrow C_{\Gamma}^0$ (resp. $(\Z/N\Z)^{\times}/D_{\Gamma} \rightarrow C_{\Gamma}^{\infty}$) given by $u \mapsto \langle u^{-1} \rangle \cdot (\Gamma \cdot 0)$ (resp. $u \mapsto \langle u^{-1} \rangle \cdot (\Gamma\cdot \infty)$) (where $\langle \cdot\rangle$ denotes the diamond operator) is a bijection. If $u \in (\Z/N\Z)^{\times}/D_{\Gamma} $, we denote by $[u]_{\Gamma}^0$ (resp. $[u]_{\Gamma}^{\infty}$) the image of $u$ in $C_{\Gamma}^0$ (resp. $C_{\Gamma}^{\infty}$). In other words, we have $[u]_{\Gamma}^0 = \Gamma \cdot \frac{c}{d}$ for some coprime integers $c$ and $d$ not divisible by $N$, and such that the image of $d$ in $(\Z/N\Z)^{\times}/D_{\Gamma} $ is $u^{-1}$. Similarly, $[u]_{\Gamma}^{\infty} = \Gamma \cdot \frac{a}{N\cdot b}$ for some coprime integers $a$ and $b$ not divisible by $N$, and such that the image of $a$ in $(\Z/N\Z)^{\times}/D_{\Gamma} $ is $u$. 

Let $\begin{pmatrix} a&b \\ c&d \end{pmatrix} \in \SL_2(\Z)$. We describe $\partial(\xi_{\Gamma}([c,d]))$ in the various cases that can happen.

\begin{itemize}
\item If $c \equiv 0 \text{ (modulo }N\text{)}$ then $a \equiv d^{-1} \text{ (modulo }N\text{)}$.
Thus, we have $\partial(\xi_{\Gamma}([c,d])) = [d]_{\Gamma}^{0} - [d]_{\Gamma}^{\infty}$.

\item If $d \equiv 0 \text{ (modulo }N\text{)}$ then we have $b \equiv -c^{-1} \text{ (modulo }N\text{)}$.
Thus, we have $\partial(\xi_{\Gamma}([c,d])) = [c]_{\Gamma}^{\infty} - [c]_{\Gamma}^{0}$.

\item If  $c\cdot d \not\equiv 0 \text{ (modulo }N\text{)}$ then we have $\partial(\xi_{\Gamma}([c,d])) = [c]_{\Gamma}^{\infty} - [d]_{\Gamma}^{\infty}$.
\end{itemize}

In particular, the set of $[c,d]$ such that $\partial(\xi_{\Gamma}([c,d])) \in \Z[C_{\Gamma}^{\infty}]$ coincides with the set of $[c,d]$ such that $c\cdot d \not\equiv 0 \text{ (modulo }N \text{)}$. Let $M_{\Gamma}^0$ be the sub-$\Zp$-module of $\Zp[\left((\Z/N\Z)^2\backslash \{(0,0)\} \right) / D_{\Gamma}]$ generated by the symbols $[c,d]$ with $c\cdot d \not\equiv 0 \text{ (modulo }N\text{)}$. 

The following statement is well-known, but we could not find a reference.
\begin{prop}\label{generation_Manin_C_0^{(p^r)}} 
The map $\xi_{\Gamma}$ induces a surjective homomorphism
$$\xi_{\Gamma}^0 : M_{\Gamma}^0 \rightarrow \tilde{H}_{\Gamma}$$
whose kernel is $R_{\Gamma}^0 = (M_{\Gamma}^0)^{\tau} + (M_{\Gamma}^0)^{\sigma} + \sum_{d \in (\Z/N\Z)^{\times}} \Zp \cdot [-d,d]$ where $(M_{\Gamma}^0)^{\tau}$ (resp. $(M_{\Gamma}^0)^{\sigma}$) is the subgroup of elements of $M_{\Gamma}^0$ fixed by the right action of $\tau$ (resp. $\sigma$).
\end{prop}
\begin{proof}
Let $\xi_{\Gamma}' = \xi_{\Gamma} \circ \kappa^{-1} :  \Zp[\left((\Z/N\Z)^2\backslash \{(0,0)\} \right) / D_{\Gamma}] \rightarrow \tilde{H}_{\Gamma}'$ and $\xi_{\Gamma}^0$ be the restriction of $\xi_{\Gamma}'$ to $M_{\Gamma}^0$. The computation of $\partial$ shows that $\xi_{\Gamma}^0$ takes values in $\tilde{H}_{\Gamma}$. 
Let $y \in \tilde{H}_{\Gamma}$. Since $\xi_{\Gamma}'$ is surjective, there is some element $x = \sum_{[c,d] \in  \left((\Z/N\Z)^2\backslash \{(0,0)\} \right) / D_{\Gamma}} \lambda_{[c,d]} \cdot [c,d] \in \Zp[\left((\Z/N\Z)^2\backslash \{(0,0)\} \right) / D_{\Gamma}]$ such that $\xi_{\Gamma}'(x)=y$. Since $\partial \xi_{\Gamma}'(x) \in \Zp[C_{\Gamma}^{\infty}]$, we have $\lambda_{[d,0]} = \lambda_{[0,d]}$ for all $d \in  \left((\Z/N\Z)^2\backslash \{(0,0)\} \right) / D_{\Gamma}$. Since $\xi_{\Gamma}'([0,d]+[d,0]) = 0$, the element $y$ is in the image of $\xi_{\Gamma}^0$. Thus, we have proved that $\xi_{\Gamma}^0$ is surjective.

Let $x = \sum_{[c,d] \in  \left((\Z/N\Z)^2\backslash \{(0,0)\} \right) / D_{\Gamma}} \lambda_{[c,d]} \cdot [c,d] - \mu_{[c,d]} \cdot [c,d] \in \Ker(\xi_{\Gamma}^0) = \Ker(\xi_{\Gamma}') \cap M_{\Gamma}^0$ with $\lambda_{[c,d]}=\lambda_{[c,d]\cdot \tau}$ and $\mu_{[c,d]}=\mu_{[c,d]\cdot \sigma}$ for all $[c,d] \in  \left((\Z/N\Z)^2\backslash \{(0,0)\} \right) / D_{\Gamma}$.  We also have $\lambda_{[d,0]} = \mu_{[d,0]}$ and $\lambda_{[0,d]} = \mu_{[0,d]}$ for all $d \in (\Z/N\Z)^{\times}/D_{\Gamma}$.
Note that for all $d \in (\Z/N\Z)^{\times}/D_{\Gamma}$, we have:
\begin{equation}\label{[-u,u]_relation}
[d,-d]=([d,0]+[0,d]+[d,-d])-([d,0]+[0,d]) \in \Ker(\xi_{\Gamma}^0) \text{ .}
\end{equation}
Hence, $x - \sum_{d \in (\Z/N\Z)^{\times}/D_{\Gamma}} \lambda_{[d,0]} \cdot [d,-d]$ belongs to $(M_{\Gamma}^0)^{\sigma} + (M_{\Gamma}^0)^{\tau}$ so $x$ has the desired form.

\end{proof}

\begin{corr}\label{surjection_homology}
The map $\tilde{H}_{\Gamma_1} \rightarrow \tilde{H}_{\Gamma_2}$ is surjective.
\end{corr}
\begin{proof}
The map $M_{\Gamma_1}^0 \rightarrow M_{\Gamma_2}^0$ is surjective. We conclude using Proposition \ref{generation_Manin_C_0^{(p^r)}}.
\end{proof}

The ring $\Lambda_{\Gamma_i}$ acts naturally on $R_{\Gamma_i}^0$, $M_{\Gamma_i}$ and $\tilde{H}_{\Gamma_i}$ (for $i=1,2$).

\begin{prop}\label{odd_modSymb_refined_Hida}
\begin{enumerate}
\item The kernel of the homomorphism $\tilde{H}_{\Gamma_1} \rightarrow \tilde{H}_{\Gamma_2}$ is $J_{\Gamma_1 \rightarrow \Gamma_2}  \cdot \tilde{H}_{\Gamma_1}$. 
\item The kernel of the homomorphism $H_{\Gamma_1} \rightarrow H_{\Gamma_2}$ is $J_{\Gamma_1 \rightarrow \Gamma_2}  \cdot H_{\Gamma_1}$.
\end{enumerate}
\end{prop}

\begin{proof}
We prove point (i).
Consider the following commutative diagram, where the rows are exact:
$$\xymatrix{
   0 \ar[r] &  R_{\Gamma_1}^0 \ar[r] \ar[d]  & M_{\Gamma_1}^0 \ar[r] \ar[d]  & \tilde{H}_{\Gamma_1}\ar[r] \ar[d]& 0  \\
    0 \ar[r] & R_{\Gamma_2}^0 \ar[r]  & M_{\Gamma_2}^0  \ar[r] & \tilde{H}_{\Gamma_2} \ar[r] & 0   
  } $$

It is clear that the kernel of the middle vertical arrow is $J_{\Gamma_1 \rightarrow \Gamma_2}\cdot M_{\Gamma_1}^0$. The cokernel of the left vertical map is zero by Proposition \ref{generation_Manin_C_0^{(p^r)}} (using $p>3$). The snake lemma concludes the proof of point (i).

We now prove point (ii). Using point (i), it suffices to show that $H_{\Gamma_1} \cap \left(J_{\Gamma_1\rightarrow \Gamma_2} \cdot \tilde{H}_{\Gamma_1}\right) = J_{\Gamma_1\rightarrow \Gamma_2}\cdot H_{\Gamma_1}$. Consider the following commutative diagram, where the rows are exact:
$$\xymatrix{
   0 \ar[r] &   H_{\Gamma_1} \ar[r] \ar[d]  & \tilde{H}_{\Gamma_1} \ar[r] \ar[d]  &  \Zp[C_{\Gamma_1}^{\infty}] ^0\ar[r] \ar[d]& 0  \\
    0 \ar[r] &  H_{\Gamma_1} \ar[r]  & \tilde{H}_{\Gamma_1}  \ar[r] &   \Zp[C_{\Gamma_1}^{\infty}]^0 \ar[r] & 0   
  }$$
Here, the vertical maps are induced by the action of $[d]-1$ where $d$ is a fixed generator of $\Ker((\Z/N\Z)^{\times}/D_{\Gamma_1} \rightarrow (\Z/N\Z)^{\times}/D_{\Gamma_2})$. Recall that $J_{\Gamma_1 \rightarrow \Gamma_2}$ is principal, generated by $[d]-1$. Thus, to prove (ii) it suffices to show (using the snake Lemma) that the map $\tilde{H}_{\Gamma_1}[J_{\Gamma_1 \rightarrow \Gamma_2}] \rightarrow \Zp[C_{\Gamma_1}^{\infty}] ^0[J_{\Gamma_1 \rightarrow \Gamma_2}]$ is surjective. 

It suffices to show that the boundary map $M_{\Gamma_1}^0[J_{\Gamma_1 \rightarrow \Gamma_2}] \rightarrow  \Zp[C_{\Gamma_1}^{\infty}] ^0[J_{\Gamma_1 \rightarrow \Gamma_2}]$ is surjective. Since we can identify $C_{\Gamma_1}^{\infty}$ with $(\Z/N\Z)^{\times}/D_{\Gamma_1}$, the action of $(\Z/N\Z)^{\times}/D_{\Gamma_1}$ on $C_{\Gamma_1}^{\infty}$ is free. Thus, any element of $\Zp[C_{\Gamma_1}^{\infty}] ^0[J_{\Gamma_1 \rightarrow \Gamma_2}]$ is of the form $\sum_{x \in C_{\Gamma_1}^{\infty}} \lambda_x \cdot (\sum_{k=0}^{m-1} [d^{k-1}])\cdot [x]$ where $m$ is the order of $d$ and $\sum_{x \in C_{\Gamma_1}^{\infty}} \lambda_x = 0$. Thus, $\Zp[C_{\Gamma_1}^{\infty}] ^0[J_{\Gamma_1 \rightarrow \Gamma_2}]$ is spanned over $\Zp$ by the elements $(\sum_{k=0}^{m-1} [d^{k-1}])\cdot ([u]-[v])$ for $u$, $v$ $\in C_{\Gamma_1}^{\infty}$. If we identify $u$ and $v$ with elements of $(\Z/N\Z)^{\times}/D_{\Gamma_1}$ and lift them to elements of $(\Z/N\Z)^{\times}$, $(\sum_{k=0}^{m-1} [d^{k-1}])\cdot ([u]-[v])$ is the boundary of the Manin symbol $(\sum_{k=0}^{m-1} [d^{k-1}])\cdot [u,v]$, which is annihilated by $J_{\Gamma_1 \rightarrow \Gamma_2}$. This concludes the proof of point (ii).
\end{proof}

\section{Eisenstein ideals of $X_1^{(p)}(N)$}\label{odd_modSymb_section_eisenstein_ideal}
We keep the notation of section \ref{odd_modSymb_section_hida} and add the following ones.
\begin{itemize}
\item $t$ is the $p$-adic valuation of $N-1$.
\item $P=(\Z/N\Z)^{\times}/\left( (\Z/N\Z)^{\times} \right)^{p^t}$.
 \item $P' =\left( (\Z/N\Z)^{\times} \right)^{p^t}$.
 \item $\Lambda^{(p)} = \Zp[P]$.
 \item $J^{(p)} \subset \Lambda^{(p)}$ is the augmentation ideal.
 \item $J_{(p)} = \Ker\left(\Zp[(\Z/N\Z)^{\times}] \rightarrow \Lambda^{(p)} \right)$.
 \item $\Gamma_1^{(p)}(N) \subset \Gamma_0(N)$ is the subgroup of $\Gamma_0(N)$ corresponding to the matrices whose diagonal entries are in $P'$ modulo $N$.
  \item If $\Gamma = \Gamma_1^{(p)}(N)$, we let $X_1^{(p)}(N)=X_{\Gamma}$, $\tilde{H}^{(p)} = \tilde{H}_{\Gamma}$, $\tilde{H}^{(p)}_+ = \left(\tilde{H}_{\Gamma}\right)_+$, $H^{(p)} = H_{\Gamma}$, $H^{(p)}_+ = \left(  H_{\Gamma} \right)_+$, $\widetilde{\mathbb{T}'}^{(p)} = \tilde{\mathbb{T}}_{\Gamma}'$, $\tilde{\mathbb{T}}^{(p)} = \tilde{\mathbb{T}}_{\Gamma}$, $\mathbb{T}^{(p)} = \mathbb{T}_{\Gamma}$, $C_0^{(p)} = C_{\Gamma}^0$ and $C_{\infty}^{(p)} = C_{\Gamma}^{\infty}$.
 \item If $\Gamma = \Gamma_0(N)$, we recall that $\tilde{H} = \tilde{H}_{\Gamma}$, $H = H_{\Gamma}$ and $\mathbb{T}= \mathbb{T}_{\Gamma}$.
 \item $\tilde{I}_{0}'$ is the ideal of $\widetilde{\mathbb{T}'}^{(p)}$ is generated by the operators $T_{n} - \sum_{d \mid n, \gcd(d,N)=1} \langle d \rangle \cdot d$. 
 \item We denote by $\tilde{I}_{0}$ (resp. $I_{0}$) the image of $\tilde{I}_{0}'$ in $\tilde{\mathbb{T}}^{(p)}$ (resp. $\mathbb{T}^{(p)}$).
\end{itemize}

The main goal of this section is to give an explicit description of $\tilde{H}^{(p)}/\tilde{I}_{0} \cdot \tilde{H}^{(p)}$. The Hecke algebra $\widetilde{\mathbb{T}'}^{(p)}$ (resp. $\tilde{\mathbb{T}}^{(p)}$, $\mathbb{T}^{(p)}$) acts faithfully on the space of modular forms of weight $2$ and level $\Gamma_1^{(p)}(N)$ (resp. which vanish at the cusps in $C_{\infty}^{(p)}$, resp. which are cuspidal). 

Let
$$ \zeta^{(p)} = \sum_{x \in (\Z/N\Z)^{\times}}  \B_2\left(\frac{x}{N}\right)\cdot [x] \in \Lambda^{(p)} $$
and
$$\nu^{(p)} =  \sum_{x \in P}  [x] \in \Lambda^{(p)} \text{ .}$$
Here, $\B_2(x) = (x-E(x))^2-(x-E(x))+\frac{1}{6}$ is the second periodic Bernoulli polynomial function ($E(x)$ is the integer part of $x$).

The following lemma will be useful in our proofs. It is an immediate consequence of Nakayama's lemma, since $\Lambda^{(p)}$ is a local ring.

\begin{lem}\label{odd_modSymb_Nakayama}
Let $f : M_1 \rightarrow M_2$ be a morphism of finitely generated $\Lambda^{(p)}$-modules. Let $\overline{f} :  M_1 \rightarrow M_2/J^{(p)}\cdot M_2$ be the map obtained from $f$. Then $f$ is surjective if and only if $\overline{f}$ is surjective.
\end{lem}

The following result is analogous to Mazur's computation of $\mathbb{T}/I$ \cite[Proposition II.9.7]{Mazur_Eisenstein}. In fact, our proof uses Mazur's results and techniques. 
\begin{thm}\label{odd_modSymb_Eisenstein_ideal_X1}
Assume $p \geq 5$.
\begin{enumerate}

\item\label{odd_modSymb_Eisenstein_ideal_X1_i} The map $\Lambda^{(p)}\rightarrow \tilde{\mathbb{T}}^{(p)}$ given by $[d] \mapsto \langle d \rangle$ gives an isomorphism of $\Lambda^{(p)}$-modules $$\Lambda^{(p)}/(\zeta^{(p)}) \xrightarrow{\sim} \tilde{\mathbb{T}}^{(p)}/\tilde{I}_{0} \text{ .}$$

\item\label{odd_modSymb_Eisenstein_ideal_X1_ii} The map $\Lambda^{(p)}\rightarrow \mathbb{T}^{(p)}$ given by $[d] \mapsto \langle d \rangle$ gives an isomorphism of $\Lambda^{(p)}$-modules $$\Lambda^{(p)}/\left(\zeta^{(p)}, \nu^{(p)}\right) \xrightarrow{\sim} \mathbb{T}^{(p)}/I_{0} \text{ .}$$

\item\label{odd_modSymb_Eisenstein_ideal_X1_iii} The groups $\tilde{\mathbb{T}}^{(p)}/\tilde{I}_{0}$ and $\mathbb{T}^{(p)}/I_{0}$ are finite.
\end{enumerate}
\end{thm}
\begin{proof}
The assertion (\ref{odd_modSymb_Eisenstein_ideal_X1_iii}) follows from (\ref{odd_modSymb_Eisenstein_ideal_X1_i}), (\ref{odd_modSymb_Eisenstein_ideal_X1_ii}) and the well-known property of Stickelberger elements (non vanishing of $L(\chi,2)$ for any even Dirichlet character $\chi$).

Let $$E_0:=\sum_{n \geq 1} \left(\sum_{d \mid n, \atop \gcd(d,N)=1}  [d]\cdot d\right)\cdot q^n  \in \Lambda^{(p)}[[q]]$$.

For any non-trivial character $\epsilon : P \rightarrow \mathbf{C}^{\times}$, the element $$-\frac{N}{4}\cdot \left(\sum_{x \in (\Z/N\Z)^{\times}} \epsilon(x)\cdot \B_2(\frac{x}{N})\right) + \sum_{n \geq 1} \left(\sum_{d \mid n, \atop \gcd(d,N)=1}  \epsilon(d) \cdot d\right)\cdot q^n \in \mathbf{C}[[q]]$$ is the $q$-expansion at the cusp $\infty$ of an Eisenstein series of weight $2$ and level $\Gamma_1^{(p)}(N)$, which we denote by $E_{1, \epsilon}$ (\cf for instance \cite[Theorem $4.6.2$]{Diamond_Shurman}). Furthermore, we have already seen that $\frac{N-1}{24} + \sum_{n \geq 1} \left(\sum_{d \mid n, \gcd(d,N)=1} d\right)\cdot q^n$ is the $q$-expansion at the cusp $\infty$ of an Eisenstein series of level $\Gamma_0(N)$, denoted by $E_2$.  Fix an embedding of $\overline{\mathbf{Q}}_p \hookrightarrow \mathbf{C}$. We get a natural injective ring homomorphism $\iota : \Lambda^{(p)}\rightarrow \prod_{\epsilon \in \hat{P}} \mathbf{C}$ where $\hat{P}$ is the set of characters of $P$. Thus, we have shown that $-\frac{N}{4}\cdot \zeta^{(p)} + E_{0} \in \Lambda^{(p)}[[q]]$ is the $q$-expansion at the cusp $\infty$ of a modular form of weight $2$ and level $\Gamma_1^{(p)}(N)$ over $\prod_{\epsilon \in \hat{P}} \mathbf{C}$. By the $q$-expansion principle \cite[Corollary 1.6.2]{Katz_properties}, such a modular form is over $\Lambda^{(p)}$. We denote it by $F_0$. Since for all $d \in P$ we have $\langle d \rangle E_{1,\epsilon} = \epsilon(d) \cdot E_{1,\epsilon}$, the $q$-expansion principle shows that
$\langle d \rangle  F_{0} = [d]\cdot F_{0}$.

We now prove point (\ref{odd_modSymb_Eisenstein_ideal_X1_i}). The map $\Lambda^{(p)} \rightarrow \tilde{\mathbb{T}}^{(p)}/\tilde{I}_0$ is surjective (by definition of $\tilde{I}_0$). Let $K$ denote its kernel. Let $\overline{E}_0$ be the image of $E_0$ in $(\Lambda^{(p)}/K)[[q]]$. Then $\overline{E}_0$ is the $q$-expansion at $\infty$ of a modular form (still denoted by $\overline{E}_0$) satisfying $\langle d \rangle \cdot \overline{E}_0 = [d] \cdot \overline{E}_0$ for all $d\in P$. Furthermore, $K$ is the largest ideal of $\Lambda^{(p)}$ satisfying this property. By the discussion above, we have proved that $K$ is the largest ideal of $\Lambda^{(p)}$ such that $-\frac{N}{4}\cdot \zeta^{(p)} \in \Lambda^{(p)}/K$ is the $q$-expansion at $\infty$ of a modular form $F$ over $\Lambda^{(p)}/K$ satisfying  $\langle d \rangle \cdot F = [d] \cdot F$ for all $d\in P$.

\begin{lem}\label{odd_modSymb_constant_modular_form}
Let $I$ be an ideal of $\Lambda^{(p)}$ and $G \in \Lambda^{(p)}/I$. Assume that $G$ is the $q$-expansion of a modular form of weight $2$ and level $\Gamma_1^{(p)}(N)$ over $\Lambda^{(p)}/I$ such that for all $d \in P$, we have $\langle d \rangle \cdot G = [d] \cdot G$. Then we have $G=0$.
\end{lem}
\begin{proof}
For simplicity, we denote $J^{(p)}$ by $J$ in this proof.
We prove by induction on $n\geq 0$ that $G \in J^n \cdot (\Lambda^{(p)}/I)$. This is true if $n=0$. Assume that this is true for some $n \geq 0$. By the $q$-expansion principle, $G$ is the $q$-expansion at the cusp $\infty$ of a modular form over the $\Z[\frac{1}{N}]$-module $J^n \cdot (\Lambda^{(p)}/I)$ (\cf \cite[Section 1.6]{Katz_properties} for the notion of a modular form over an abelian group). Let $\overline{G}$ be the image $G$ by the map $J^n \cdot (\Lambda^{(p)}/I) \twoheadrightarrow J^n \cdot (\Lambda^{(p)}/I)/J^{n+1} \cdot (\Lambda^{(p)}/I) $. The diamond operators act trivially on $\overline{G}$. Thus, $\overline{G}$ is the $q$-expansion at the cusp $\infty$ of a modular form of weight $2$ and level $\Gamma_0(N)$ with coefficients in the module $J^n \cdot (\Lambda^{(p)}/I)/J^{n+1} \cdot (\Lambda^{(p)}/I) $. Note that $J^n \cdot (\Lambda^{(p)}/I)/J^{n+1} \cdot (\Lambda^{(p)}/I) $ is a quotient of $J^n/J^{n+1} \simeq \Z/p^t\Z$. Since $p \geq 5$ and $\gcd(N,p)=1$, \cite[Lemma 5.9, Corollary 5.11]{Mazur_Eisenstein} shows that $\overline{G}=0$, \ie we have $G \in J^{n+1} \cdot (\Lambda^{(p)}/I)$. This concludes the induction step. Since $\bigcap_{n \geq 0} J^n \cdot (\Lambda^{(p)}/I) = 0$, we have $G=0$. This concludes the proof of Lemma \ref{odd_modSymb_constant_modular_form}.
\end{proof}
By Lemma \ref{odd_modSymb_constant_modular_form}, we have $F=0$. This proves that $K = (-\frac{N}{4}\cdot \zeta^{(p)})=(\zeta^{(p)})$, which concludes the proof of Theorem \ref{odd_modSymb_Eisenstein_ideal_X1} (\ref{odd_modSymb_Eisenstein_ideal_X1_i}).

We finally prove Theorem \ref{odd_modSymb_Eisenstein_ideal_X1} (\ref{odd_modSymb_Eisenstein_ideal_X1_ii}). Let $w_N$ be the Atkin--Lehner involution. 
By \cite[Proposition 1]{Weisinger_thesis}, the $q$-expansion at the cusp $\infty$ of $w_N(E_{1, \epsilon})$ is
$$\left(\sum_{x \in (\Z/N\Z)^{\times}} \epsilon(x) \cdot e^{\frac{2 i \pi x }{N}} \right) \cdot \left(\sum_{n \geq 1} \sum_{d \mid n \atop \gcd(d,N)=1}  \epsilon(d)^{-1} \cdot \frac{n}{d} \cdot q^n \right) \in \Lambda^{(p)}[[q]] \text{ .}$$
Furthermore, we have $w_N(E_2) = -E_2$.

Let $\mu \in \overline{\Z}_p$ be the primitive $N$th root of unity corresponding to $e^{\frac{2 i \pi}{N}}$ under our fixed embedding $\overline{\mathbf{Q}}_p \hookrightarrow \mathbf{C}$. Let $\Lambda^{(p)}{}' = (\Zp[\mu])[P]$ and $\mathcal{G'} = \sum_{x \in (\Z/N\Z)^{\times}} \mu^x \cdot [x] \in \Lambda^{(p)}{}'$. The element $\mathcal{G}'$ is invertible in $\Lambda^{(p)}{}'$ since its degree is $-1$, which is prime to $p$, and $\Lambda^{(p)}{}'$ is a local ring whose maximal ideal is $J'+(\varpi)$ where $J'$ is the augmentation ideal of $\Lambda^{(p)}{}'$ and $\varpi$ is a uniformizer of $\Zp[\mu]$.

By reformulating Weisinger's formula, the $q$-expansion principle shows that the $q$-expansion at the cusp $\infty$ of $F_{\infty}:=w_N(F_{0})$ is 
$$
-\frac{N-1}{24\cdot p^t}\cdot \nu^{(p)} + \mathcal{G}' \cdot \sum_{n\geq 1} \left(\sum_{d \mid n, \text{ gcd}(d,N)=1}  [d]^{-1} \cdot \frac{n}{d} - a_n \right)\cdot q^n \text{ , }
$$
where if $n=N^v\cdot n_0$ with $v\in \mathbf{Z}_{\geq 0}$ and $\gcd(n_0, N)=1$, we let 
$$a_n = \nu^{(p)}\cdot \frac{N^v-1}{p^t} \cdot \sum_{d \mid n_0} d \in \Lambda^{(p)}\text{ .}$$

The modular form $F_0$ is cuspidal modulo some ideal $I$ of $\Lambda^{(p)}$ if and only if $I$ contains $\zeta^{(p)}$ and $a_0(F_{\infty})$. This concludes the proof of Theorem \ref{odd_modSymb_Eisenstein_ideal_X1}.
\end{proof}

\section{The extended winding homomorphism}
In the following result, we extend the winding homomorphism of Mazur \cite[p. 137]{Mazur_Eisenstein}.

\begin{thm}\label{odd_modSymb_structure_H_1}
\begin{enumerate}
\item\label{odd_modSymb_structure_H_1_i} The $\tilde{\mathbb{T}}^{(p)}/\tilde{I}_0$-module $\tilde{I}_0/\tilde{I}_0^2$ is free of rank $1$, \ie $\tilde{I}_0$ is locally principal. Consequently, the completion of $\tilde{\mathbb{T}}^{(p)}$ at $\tilde{I}_0$ is complete intersection.

\item \label{odd_modSymb_structure_H_1_ii} There is a canonical group isomorphism $$\tilde{I}_0/\tilde{I}_0^2 \xrightarrow{\sim} \tilde{H}^{(p)}_+/\tilde{I}_{0} \cdot  \tilde{H}^{(p)}_+ \text{ .}$$
In particular, the $\Lambda^{(p)}$-module $\tilde{H}^{(p)}_+/\tilde{I}_{0} \cdot  \tilde{H}^{(p)}_+$ is (non-canonically) isomorphic to $\Lambda^{(p)}/(\zeta^{(p)})$. 
\item \label{odd_modSymb_structure_H_1_iii} The $\Lambda^{(p)}$-module $H^{(p)}_+/I_0 \cdot  H^{(p)}_+$ is isomorphic to $\Lambda^{(p)}/(\zeta^{(p)}, \nu^{(p)})$. 
\end{enumerate}
\end{thm}
\begin{proof} 
We first define a homomorphism of $\Lambda^{(p)}$-modules $\tilde{e}: \tilde{I}_{0} \rightarrow \tilde{H}^{(p)}_+$ as follows.

We have a map $\tilde{I}_{0}' \rightarrow \tilde{H}^{(p)}_+$
given by $\eta \mapsto \eta\cdot \{0,\infty\}$. 
This induces the desired map $\tilde{I}_{0} \rightarrow \tilde{H}^{(p)}_+$. 
Indeed, if $\eta \in \tilde{I}_{0}'$ maps to $0$ in $\tilde{I}_{0}$ then $\eta$ annihilates all the Eisenstein series of $M_2(\Gamma_1^{(p)}(N), \mathbf{C})$. There is an Eisenstein series $E \in M_2(\Gamma_1^{(p)}(N), \mathbf{C})$ such that the divisor of the meromorphic differential form $E(z)dz$ is $(0)-(\infty)$. This Eisenstein series induces (via integration) a morphism $H_1(Y_1^{(p)}(N), \Z) \rightarrow \mathbf{C}$. By intersection duality, we get an element $\mathcal{E} \in H_1(X_1^{(p)}(N), C_0^{(p)}\cup C_{\infty}^{(p)}, \mathbf{C})$. Since $E$ is annihilated by $\eta^*:=w_N\eta w_N^{-1}$, the element $\mathcal{E}$ is annihilated by $\eta$. Since $\mathcal{E} - \{0,\infty\} \in H_1(X_1^{(p)}(N), \mathbf{C})$ and $\eta$ acts trivially on $H_1(X_1^{(p)}(N), \mathbf{C})$, we see that $\eta \cdot \{0,\infty\}=0$. 

Let $e: I \rightarrow H_+$ be the winding homomorphism of Mazur, denoted by $e_+$ in \cite[Definition, p. 137]{Mazur_Eisenstein}. We denote by a bold letter the various $\tilde{\mathbb{T}}^{(p)}$ or $\mathbb{T}$-modules involved completed at $\tilde{I}_{0}$ or at $I$. Thus, for example, $\tilde{\textbf{H}}^{(p)}_+$ (resp. $\textbf{H}_+$) is the $\tilde{I}_0$ (resp. $I$)-adic completion of $\tilde{H}^{(p)}_+$ (resp. $H_+$). Let $\tilde{\textbf{e}}: \tilde{\textbf{I}}_{0} \rightarrow \tilde{\textbf{H}}^{(p)}_+$ (resp. $\textbf{e} : \textbf{I} \rightarrow \textbf{H}_+ $) be the map obtained after completion at the ideal $\tilde{I}_{0}$ (resp. $I$).

\begin{lem}\label{odd_modSymb_completion_projection}
The map $\tilde{H}^{(p)}_+ \rightarrow H_+$ defines by passing to completion a group isomorphism $\tilde{\textbf{H}}^{(p)}_+/J\cdot \tilde{\textbf{H}}^{(p)}_+ \xrightarrow{\sim} \textbf{H}_+$.
\end{lem}
\begin{proof}
The rings $\tilde{\mathbb{T}}^{(p)}$ and $\mathbb{T}$ are semi-local and $p$-adically complete, so we have $$\tilde{\mathbb{T}}^{(p)} = \bigoplus_{\tilde{\mathfrak{m}} \in \text{SpecMax}(\tilde{\mathbb{T}}^{(p)})} (\tilde{\mathbb{T}}^{(p)})_{\tilde{\mathfrak{m}}}$$ and $$\mathbb{T} = \bigoplus_{\mathfrak{m} \in \text{SpecMax}(\mathbb{T}) } \mathbb{T}_{\mathfrak{m}}$$ where the subscript means the completion. By Theorem \ref{odd_modSymb_Eisenstein_ideal_X1} (\ref{odd_modSymb_Eisenstein_ideal_X1_i}), there is a unique maximal ideal $\tilde{\mathfrak{m}}_0 \in \text{SpecMax}(\tilde{\mathbb{T}}^{(p)})$ containing $\tilde{I}_0$. Similarly, there exists a unique maximal ideal $\mathfrak{m} \in \text{SpecMax}(\mathbb{T})$ containing Mazur's Eisenstein ideal $I$. We denote by $e_{\tilde{\mathfrak{m}}_0}$ (resp. $e_{\mathfrak{m}}$) the idempotent of $\tilde{\mathbb{T}}^{(p)}$ (resp. $\mathbb{T}$) corresponding to $\tilde{\mathfrak{m}}_0$ (resp. $\mathfrak{m}$). The image of $e_{\tilde{\mathfrak{m}}_0}$ in $\mathbb{T}$ is $e_{\mathfrak{m}}$. We have $\tilde{\textbf{H}}^{(p)}_+ = e_{\tilde{\mathfrak{m}}_0}\cdot \tilde{H}^{(p)}_+$ and $\textbf{H}_+ = e_{\mathfrak{m}} \cdot H_+$. Since $\tilde{H}^{(p)}_+/J\cdot \tilde{H}^{(p)}_+ = H_+$ by Proposition \ref{odd_modSymb_refined_Hida}, we have:
$$\tilde{\textbf{H}}^{(p)}_+/J\cdot \tilde{\textbf{H}}^{(p)}_+ =e_{\tilde{\mathfrak{m}}_0} \cdot (\tilde{H}^{(p)}_+/J\cdot \tilde{H}^{(p)}_+)  = e_{\mathfrak{m}} \cdot H_+ = \textbf{H}_+ \text{ .} $$
This concludes the proof of Lemma \ref{odd_modSymb_completion_projection}.
\end{proof}

Thus, we get a commutative diagram:
$$\xymatrix{
 \tilde{\textbf{I}}_{0} \ar[r]^{\tilde{\textbf{e}}} \ar[d]  & \tilde{\textbf{H}}^{(p)}_+  \ar[d]  \\
     \textbf{I} \ar[r]^{\textbf{e}} & \textbf{H}_+ }$$

Since $\textbf{e}$ is surjective (it is even an isomorphism by \cite[Theorem $18.10$]{Mazur_Eisenstein}) and the map $\tilde{\textbf{I}}_{0}\rightarrow \textbf{I}$ is surjective, Lemmas \ref{odd_modSymb_Nakayama} and \ref{odd_modSymb_completion_projection} show that $\tilde{\textbf{e}}$ is surjective. By the Eichler--Shimura isomorphism (over $\mathbf{C}$), the $\Zp$-rank of these two modules must be equal to the $\Zp$-rank of $\tilde{\mathbf{T}}^{(p)}$. Thus, $\tilde{\textbf{e}}$ is an isomorphism. 

By passing to the quotient map, $\tilde{\textbf{e}}$ gives rise to an isomorphism of $\Lambda^{(p)}$-modules $\tilde{I}_{0} / \tilde{I}_{0}^2\simeq \tilde{H}^{(p)}_+/\tilde{I}_{0}\cdot \tilde{H}^{(p)}_+$. The $\Lambda^{(p)}$-module $\tilde{H}^{(p)}_+/\tilde{I}_{0}\cdot \tilde{H}^{(p)}_+$, and so $\tilde{I}_{0} / \tilde{I}_{0}^2$, is cyclic since it is cyclic modulo $J$ by Proposition \ref{odd_modSymb_refined_Hida}. By Nakayama's Lemma the ideal $\tilde{\textbf{I}}_{0}$ is principal. Since a generator of $\tilde{\textbf{I}}_{0}$ is not a zero-divisor, we get:
$$\tilde{\mathbb{T}}^{(p)}/\tilde{I}_{0} \simeq \tilde{I}_{0}/\tilde{I}_{0}^2 \text{ .}$$
This concludes the proof of points (\ref{odd_modSymb_structure_H_1_i}) and (\ref{odd_modSymb_structure_H_1_ii}) by Theorem \ref{odd_modSymb_Eisenstein_ideal_X1} (\ref{odd_modSymb_Eisenstein_ideal_X1_i}), except for the assertion concerning the complete intersection property. Since the ring homomorphism $\Lambda^{(p)} \rightarrow \tilde{\mathbf{T}}^{(p)}$ is injective, the $\Zp$-algebra $\tilde{\mathbf{T}}^{(p)}$ is isomorphic to $\Lambda^{(p)}[X]/(P(X))$ for some $P \in \Lambda^{(p)}[X]$ such that $P(0)=\zeta^{(p)}$. Thus, $\tilde{\mathbf{T}}^{(p)} \simeq \Zp[X,Y]/((1+Y)^{p^t}-1, Q(X,Y))$ for some $Q(X,Y) \in \Zp[X,Y]$ satisfying $Q(0,Y)=\sum_{a\in (\Z/N\Z)^{\times}} \B_2(\frac{a}{N})\cdot (1+Y)^{\log(a)}$ (we take any representative of $\log(a)$ in $\Z$). In particular, for any $p^t$th root of unity $\mu$, we have $Q(0,\mu-1)\neq 0$. Thus, $Q(0,Y)$ and $R(Y):=(1+Y)^{p^t}-1$ are coprime in $\Zp[Y]$. We easily deduce that $Q(X,Y)$ is not a zero-divisor in $\Zp[X,Y]/(R(Y))$. Thus, the sequence $(R,Q)$ is regular in $\Zp[X,Y]$ so $\tilde{\mathbf{T}}^{(p)}$ is complete intersection.

We finally prove point (\ref{odd_modSymb_structure_H_1_iii}). By Theorem \ref{odd_modSymb_structure_H_1} (\ref{odd_modSymb_structure_H_1_iii}), it suffices to prove that $H^{(p)}_+$ is free of rank $1$ over $\mathbb{T}^{(p)}$. The $\mathbb{T}$-module $I\cdot H_+$ is free of rank $1$. Indeed, it suffices to prove this by localizing at each maximal ideal of $\mathbb{T}$; at non-Eisenstein ideals this follows from \cite[Corollary 15.2]{Mazur_Eisenstein} since $p>2$ while at the maximal Eisenstein ideal this follows from \cite[Proposition 16.6]{Mazur_Eisenstein}. Proposition \ref{odd_modSymb_refined_Hida} and Nakayama's lemma then show that $H^{(p)}_+$ is monogenic over $\mathbb{T}^{(p)}$. Since the $\Zp$ rank of $\mathbb{T}^{(p)}$ and $H^{(p)}_+$ are the same (namely the genus of $X_1(N)^{(p)}$), $H^{(p)}_+$ must be free of rank $1$ over $\mathbb{T}^{(p)}$.
\end{proof}

The following result will be useful later.

\begin{prop} \label{odd_modSymb_construction}
The projection map $\tilde{H}^{(p)} \rightarrow H$ gives an isomorphism
$$ H^{(p)}_+/(I_0 + J)\cdot H^{(p)}_+ \xrightarrow{\sim} I\cdot H_+/I^2\cdot H_+\text{ .}$$
\end{prop}
\begin{proof}
By \cite[Lemma II.18.7]{Mazur_Eisenstein}, the image of the map $H^{(p)}_+ \rightarrow H_+$ is $I\cdot H_+$. Thus, the image of the map $I_0\cdot H^{(p)}_+ \rightarrow H_+$ is $I^2\cdot H_+$.
Proposition \ref{odd_modSymb_construction} then follows from Proposition \ref{odd_modSymb_refined_Hida}.
\end{proof}

\section{Sharifi's conjecture for \texorpdfstring{$X_{1}(N)$}{Lg} and  \texorpdfstring{$X_{1}(N)^{(p)}$}{Lg}.}\label{section_Sharifi_conj}
Keep the notation of the previous sections. In this section, we discuss Sharifi's conjecture (Conjecture \ref{Sharifi_conjecture}) and prove Theorem \ref{thm_conj}.

For simplicity, if $u,v\in \Z/N\Z$ with $\gcd(u,v)=1$ we let
$$[u,v]^{*}:=\xi_{\Gamma_1(N)}([u,v]) \in \tilde{H}_{\Gamma_1(N)} $$
and
$$[u,v]^{*}_+:=\frac{1}{2}\cdot([u,v]^{*}+[-u,v]^{*}) \in (\tilde{H}_{\Gamma_1(N)})_+ \text{ .}$$
Note that $[u,v]^{*}_+$ was denoted by $\xi_1([(u,v)])$ in \S \ref{introduction}. We shall abuse notation and still denote by $[u,v]^{*}$ (resp. $[u,v]^{*}_+$) the image of $[u,v]^{*}$ in $\tilde{H}^{(p)}$ (resp. $\tilde{H}^{(p)}_+$), when there are no possible confusions.

\begin{prop}[Sharifi]\label{Prop_existence_varpi}
The homomorphism of \S \ref{introduction}
$$\tilde{\varpi}: (\tilde{H}_{\Gamma_1(N)})_+ \rightarrow \mathcal{K}$$
 given by $$[u,v]^{*}_+\mapsto \langle 1-\zeta_{N}^u,1-\zeta_{N}^v \rangle$$ is well-defined.
\end{prop}
\begin{proof}
We need to check that the map $[u,v]^{*}_+\mapsto \langle 1-\zeta_{N}^u,1-\zeta_{N}^v \rangle$ satisfies the Manin relations. The proof is identical to the one of \cite[Proposition 5.7]{Sha1}.
\end{proof}

A straightforward computation shows that $\tilde{\varpi}$ is $\Lambda$-equivariant (recall the convention for the action of $\Lambda$ on $\mathcal{K}$ in the introduction). We denote by 
$$\varpi: (H_{\Gamma_1(N)})_+\rightarrow \mathcal{K}$$
the restriction of $\tilde{\varpi}$ to $(H_{\Gamma_1(N)})_+ =H_1(X_1(N), \Zp)_+$.

We now describe an analogue of Sharifi's conjecture for the modular curve $X_1(N)^{(p)}$. Let $\zeta_N^{(p)} \in \Q(\zeta_N)$ be such that $[\Q(\zeta_N^{(p)}):\Q]=p^t$ (where $t$ is the $p$-adic valuation of $N-1$). We let $\mathcal{K}^{(p)} := K_2(\Z[\zeta_N^{(p)}, \frac{1}{Np}])$. The cyclotomic character gives an isomorphism $\Gal(\Q(\zeta_N)/\Q)\xrightarrow{\sim} (\Z/N\Z)^{\times}$. Under this identification, we have $\Gal(\Q(\zeta_N^{(p)})/\Q)=P$ and $\Gal(\Q(\zeta_N)/\Q(\zeta_N^{(p)}))=P'$. We thus have a canonical action of $\Lambda^{(p)}$ on $\mathcal{K}^{(p)}$.

\begin{rem}\label{comparison_K_etale}
The \'etale Chern class maps gives canonical isomorphisms $$H^2_{\text{\'et}}(\Z[\zeta_N, \frac{1}{Np}], \Zp(2)) \xrightarrow{\sim} \mathcal{K}$$
and
$$H^2_{\text{\'et}}(\Z[\zeta_N^{(p)}, \frac{1}{Np}], \Zp(2)) \xrightarrow{\sim} \mathcal{K}^{(p)} \text{ .}$$
Under these identifications, the norm map $\mathcal{K} \rightarrow \mathcal{K}^{(p)}$ corresponds to the corestriction map $H^2_{\text{\'et}}(\Z[\zeta_N, \frac{1}{Np}], \Zp(2)) \rightarrow H^2_{\text{\'et}}(\Z[\zeta_N^{(p)}, \frac{1}{Np}], \Zp(2))$. We will use these identifications freely in the rest of the article.
\end{rem}

We collect a few useful facts about our various $K$-groups.

\begin{prop}\label{Prop_K_theory}
\begin{enumerate}
    \item\label{Prop_K_theory_norm} The norm map $\mathcal{K}\rightarrow \mathcal{K}^{(p)}$ induces isomorphisms $\mathcal{K}/J_{(p)}\cdot\mathcal{K} \xrightarrow{\sim} \mathcal{K}^{(p)}$ and $$J\cdot \mathcal{K}/J^2\cdot \mathcal{K} \xrightarrow{\sim} J^{(p)}\cdot \mathcal{K}^{(p)}/(J^{(p)})^2\cdot \mathcal{K}^{(p)} \text{ .}$$
    \item\label{Prop_K_theory_augmentation} We have a group isomorphism $$\mathcal{K}/J\cdot \mathcal{K}\xrightarrow{\sim} (\Z/N\Z)^{\times} \otimes_{\Z} \Zp$$
    given by the residue symbol $\langle x, y \rangle \otimes 1 \mapsto \overline{\frac{x^{v(y)}}{y^{v(x)}}} \otimes 1$
    where $v(\cdot)$ is the $(1-\zeta_N)$-adic valuation and the bar means reduction modulo $(1-\zeta_N)$.
    \item\label{Prop_K_theory_Lambda} The $\Lambda^{(p)}$-module $\mathcal{K}^{(p)}$ is isomorphic to $\Lambda^{(p)}/(\zeta^{(p)})$ (recall the choice of the $\Lambda^{(p)}$-action made in the introduction).
\end{enumerate}
\end{prop}
\begin{proof}
Proof of point (\ref{Prop_K_theory_norm}). 
By \cite[Propositions 8.3.18 and 3.3.11]{Neukirch_Galois}, the corestriction induces an isomorphism
$$H^2_{\text{\'et}}(\Z[\zeta_N, \frac{1}{Np}], \Zp(2))/J_{(p)}\cdot H^2_{\text{\'et}}(\Z[\zeta_N, \frac{1}{Np}], \Zp(2)) \xrightarrow{\sim} H^2_{\text{\'et}}(\Z[\zeta_N^{(p)}, \frac{1}{Np}], \Zp(2))\text{ .}$$
This proves the first assertion by Remark \ref{comparison_K_etale}. The second assertion follows from the first and from the fact that $J_{(p)}\cdot J \subset J^2$.

Proof of point (\ref{Prop_K_theory_augmentation}). Again by \cite[Propositions 8.3.18 and 3.3.11]{Neukirch_Galois}, we have a canonical group isomorphism 
$$\mathcal{K}/J\cdot \mathcal{K}\xrightarrow{\sim} K_2(\Z[\frac{1}{Np}])\otimes_{\Z} \Zp \text{ .}$$
Since $K_2(\Z[\frac{1}{Np}])\otimes_{\Z} \Zp \simeq \mathbf{F}_N^{\times}\otimes_{\Z} \Zp$, the residue symbol is an isomorphism.

\iffalse
{\color{blue}{
	Can you check if the formulation below is OK?}}
\fi	
Proof of point (\ref{Prop_K_theory_Lambda}). This is a consequence of the work of Greither and Popescu on the Coates-Sinnott conjecture. Let $K/k$ be abelian extension of number fields of Galois group $\mathfrak{G}$, and let $S$ be a finite set of primes in $k$ that contains all the ramified primes and the set $S_{\infty}$ of infinite places. 

Let $n\geq 2$ be an integer and $k$ be a number field. If $k$ is totally real, we let $e_{n}(K/k):=\prod_{v\in S_{\infty}}\frac{(1+(-1)^n\sigma_{v})}{2}$, where $\sigma_{v}$ is a generator of the decomposition group at $v$. If $k$ is not totally real, we let $e_{n}(K/k):=0$ . We let $\Theta_{S,K/k}(1-n)=\sum_{\chi\in\widehat{\mathfrak{G}}}L_{S}(\chi^{-1},1-n)\cdot e_{\chi}$ where $e_{\chi}=\tfrac{1}{|\mathfrak{G}|}\sum_{\sigma\in\mathfrak{G}}\chi(\sigma)[\sigma^{-1}] \in \C[\mathfrak{G}]$ and $L_{S}(\chi^{-1},1-n)$ is the L-function with the Euler factors at $S$ removed. Siegel proved that $\Theta_{S,K/k}(1-n) \in \Q[\mathfrak{G}]$.

\begin{thm}\cite[Theorem 6.11]{GP}\label{G&P}
	If $k$ is totally real, $p>2$ and the classical Iwasawa $\mu$-invariant associated to the maximal CM-subfield of $K(\mu_{p})$ is zero, then we have in $\Zp[\mathfrak{G}]$:
	$$
	\textnormal{Ann}_{\Zp[\mathfrak{G}]}(H^1(\O_{K,S}[\tfrac{1}{p}],\Zp(n))_{\textnormal{tors}})\cdot\Theta_{S,K/k}(1-n)=e_{n}(K/k)\cdot \textnormal{Fitt}(H^2(\O_{K,S}[\tfrac{1}{p}],\Zp(n)),
	$$
	where $\textnormal{Fitt}$ (resp. $\textnormal{Ann}$) designates the Fitting ideal (resp. the annihilator) with respect to the ring $\Zp[\mathfrak{G}]$.
\end{thm}

We apply this result to $k=\Q$, $K=\Q(\zeta_{N}^{(p)})$, $n=2$ and $S=\{N,\infty\}$. Since $K$ is abelian, the $\mu$-invariant vanishes. By \cite[Lemma 6.9]{GP}, we have
\[H^1(\O_{K,S}[\tfrac{1}{p}],\Zp(2))_{\textnormal{tors}}=(\Qp/\Zp(2))^{G_{K}}=0.\]

Since $K=\Q(\zeta_{N}^{(p)})$ is totally real, $e_{2}(K/k)$ acts on the Fitting ideal trivially, so we get
\[\textnormal{Fitt}(H^2(\Z[\zeta_{N}^{(p)},\tfrac{1}{Np}],\Zp(2)))=\Theta_{S,K/k}(-1).\]

In the case we are interested in, we have 
\[\Theta_{S,K/k}(1-2)=\frac{N}{2}\cdot \sum_{a \in (\Z/N\Z)^{\times}} \B_{2}(\frac{a}{N})[a^{-1}].\]
If we consider instead the ``inverse'' action of $\mathfrak{G}$ on our various Galois cohomology groups, we get by Theorem \ref{G&P} \[\textnormal{Fitt}(\mathcal{K}^{(p)}) = \textnormal{Fitt}(H^2(\Z[\zeta_{N}^{(p)},\tfrac{1}{Np}],\Zp(2)))=\zeta^{(p)}.\] 

By Nakayama's lemma and Proposition \ref{Prop_K_theory} (\ref{Prop_K_theory_augmentation}),  $\mathcal{K}^{(p)}$ is a cyclic $\Lambda^{(p)}$-module, so we have $\mathcal{K}^{(p)} \simeq \Lambda^{(p)}/(\zeta^{(p)})$.
\end{proof}

\begin{rem}
Proposition \ref{Prop_K_theory} shows that we have canonical group isomorphisms $$J\cdot \mathcal{K}/J^2\cdot \mathcal{K}\xrightarrow{\sim} J/J^2\cdot \otimes_{\Zp} \mathcal{K}/J\cdot \mathcal{K}\xrightarrow{\sim} ((\Z/N\Z)^{\times})^{\otimes 2}\otimes_{\Z} \Zp $$
and that an element $\sum_{(a,b)\in ((\Z/N\Z)^{\times})^{2}} \lambda_{a,b}\cdot \langle 1-\zeta_N^a, 1-\zeta_N^b \rangle$ of $\mathcal{K}$ belongs to $J\cdot \mathcal{K}$ if and only if $\sum_{(a,b)\in ((\Z/N\Z)^{\times})^{2}} ab^{-1} \otimes \lambda_{a,b}=0$ in $(\Z/N\Z)^{\times} \otimes_{\Z} \Zp$. However, we were not able to determine explicitly the image of  $\sum_{(a,b)\in ((\Z/N\Z)^{\times})^{2}} \lambda_{a,b}\cdot \langle 1-\zeta_N^a, 1-\zeta_N^b \rangle$ in $((\Z/N\Z)^{\times})^{\otimes 2}\otimes_{\Z} \Zp$.  Thus, Conjecture \ref{conjecture_I/I^2} is not enough to answer the question of Merel mentioned in the introduction.

\end{rem}

By Propositions \ref{odd_modSymb_refined_Hida} and \ref{Prop_K_theory}, taking the $J^{(p)}$ coinvariants of the map $\tilde{\varpi}$ induces a $\Lambda^{(p)}$-equivariant map
$$\tilde{\varpi}^{(p)} : \tilde{H}_+^{(p)} \rightarrow \mathcal{K}^{(p)} \text{ .}$$
We denote by $\varpi^{(p)}$ the restriction of $\tilde{\varpi}^{(p)}$ to $H^{(p)}_+=H_1(X_1(N)^{(p)}, \Zp)_+$.

\begin{prop}\label{Prop_iso_H_K}
Assume that Conjecture \ref{Sharifi_conjecture} holds. Then 
$\tilde{\varpi}^{(p)}$ induces isomorphisms $$\tilde{H}_+^{(p)}/\tilde{I}_0\cdot \tilde{H}_+^{(p)} \xrightarrow{\sim} \mathcal{K}^{(p)}$$
and
$$H_+^{(p)}/I_0\cdot H_+^{(p)} \xrightarrow{\sim} J^{(p)}\cdot \mathcal{K}^{(p)} \text{ .}$$
\end{prop}
\begin{proof}
Under our assumption, the map $\tilde{\varpi}^{(p)}$ induces a map $\varphi : \tilde{H}_+^{(p)}/\tilde{I}_0\cdot \tilde{H}_+^{(p)} \rightarrow \mathcal{K}^{(p)}$. To prove that $\varphi$ is an isomorphism, it suffices to prove that it is surjective and that $\tilde{H}_+^{(p)}/\tilde{I}_0\cdot \tilde{H}_+^{(p)}$ and $\mathcal{K}^{(p)}$ have the same (finite) cardinality. 

To prove that $\varphi$ is surjective, it suffices to prove that its reduction $\overline{\varphi}$ modulo $J^{(p)}$ is surjective. By Proposition \ref{Prop_K_theory} (\ref{Prop_K_theory_norm}) and (\ref{Prop_K_theory_augmentation}), the map $\overline{\varphi} : \tilde{H}_+^{(p)}/\tilde{I}_0\cdot \tilde{H}_+^{(p)} \rightarrow \mathcal{K}^{(p)}/J^{(p)}\cdot \mathcal{K}^{(p)} \simeq (\Z/N\Z)^{\times} \otimes_{\Z} \Zp$ is given by $[u,v]^*_+ \mapsto uv^{-1} \otimes 1$. This proves that $\varphi$ is surjective.

By Theorem \ref{odd_modSymb_structure_H_1} and Proposition \ref{Prop_K_theory} (\ref{Prop_K_theory_Lambda}), both $\tilde{H}_+^{(p)}/\tilde{I}_0\cdot \tilde{H}_+^{(p)}$ and $\mathcal{K}^{(p)}$ are $\Lambda^{(p)}$-modules isomorphic to $\Lambda^{(p)}/(\zeta^{(p)})$. Thus, $\varphi$ is an isomorphism.

We now prove the second assertion. We have the following compatibility between the ``arithmetic boundary'' and ``topological boundary''.

\begin{lem}\label{Lemma_boundary_compatibility}
We have a commutative diagram
\begin{equation}\label{tamediagram}
\begin{tikzcd}
\tilde{H}^{(p)}_+\ar[r,"\partial"]\ar[d,"\tilde{\varpi}^{(p)}"]& \Zp[C_{\infty}^{(p)}]^0 \ar[d,"t"]\\
 \mathcal{K}^{(p)}\ar[r,"\partial'"]& (\Z/N\Z)^{\times}\otimes\Zp
\end{tikzcd}
\end{equation}
where $\partial$ is the boundary (\cf \S \ref{odd_modSymb_section_hida}), $\partial'$ is the residue symbol and $t : [u]_{\Gamma_1(N)^{(p)}}^{\infty} \mapsto u \otimes 1$ for $u \in (\Z/N\Z)^{\times}$.
\end{lem}
\begin{proof}
For $u$, $v$ $\in (\Z/N\Z)^{\times}$ we have:
$$(t\circ \partial)([u,v]^*_+)=uv^{-1} \otimes 1 = (\partial' \circ \tilde{\varpi}^{(p)})([u,v]^*_+)$$
(using the description of $\partial$ given in \S \ref{odd_modSymb_section_hida}).
\end{proof}
By Proposition \ref{Prop_K_theory} and Lemma \ref{Lemma_boundary_compatibility}, the map $\varpi^{(p)}$ takes values in $J^{(p)}\cdot \tilde{K}^{(p)}$. For any prime $\ell \neq N$, we have $(T_{\ell}-\ell-\langle \ell \rangle)\cdot \tilde{H}^{(p)} \subset H^{(p)}$ since $T_{\ell}-\ell-\langle \ell \rangle$ annihilates $C_{\infty}^{(p)}$. Since we assume that Conjecture \ref{Sharifi_conjecture} holds, for all $x \in \tilde{H}^{(p)}_+$ we have 
$$\tilde{\varpi}^{(p)}((T_{\ell}-\ell-\langle \ell \rangle)(x))= \tilde{\varpi}^{(p)}((\ell\cdot\langle \ell \rangle+1  -\ell-\langle \ell \rangle)(x)) = ([ \ell ]-1)\cdot (\ell-1) \cdot \tilde{\varpi}^{(p)}(x) \text{ .}$$
When $\ell$ varies through the primes $\neq N$, the elements $([ \ell ]-1)\cdot (\ell-1)$ generate $J^{(p)}$ (by Dirichlet's theorem). Since $\tilde{\varpi}^{(p)}$ is surjective, this proves that we have a surjective map $\varpi^{(p)}  : H_+^{(p)}/I_0\cdot H_+^{(p)} \rightarrow J^{(p)}\cdot \mathcal{K}^{(p)}$.

By Theorem \ref{odd_modSymb_structure_H_1} (\ref{odd_modSymb_structure_H_1_iii}) and Proposition \ref{Prop_K_theory} (\ref{Prop_K_theory_Lambda}), both $H_+^{(p)}/I_0\cdot H_+^{(p)}$ and $J^{(p)}\cdot \mathcal{K}^{(p)}$ are isomorphic to $\Lambda^{(p)}/(\zeta^{(p)}, \nu^{(p)})$ as $\Lambda^{(p)}$-modules. This proves that $\varpi^{(p)}: H_+^{(p)}/I_0\cdot H_+^{(p)} \rightarrow J^{(p)}\cdot \mathcal{K}^{(p)}$ is an isomorphism.
\end{proof}

We now prove Theorem \ref{thm_conj}. 
\begin{proof}
By Proposition \ref{Prop_iso_H_K}, we have a group isomorphism
$$H^{(p)}_+/(I_0+J^{(p)})\cdot H^{(p)}_+ \xrightarrow{\sim} J^{(p)}\cdot \mathcal{K}^{(p)}/(J^{(p)})^2\cdot \mathcal{K}^{(p)}$$
sending the class of $\sum \lambda_{u,v}\cdot [u,v]^*_+$ to the class of $\sum \lambda_{u,v}\cdot \langle 1-\zeta_N^u, 1-\zeta_N^v\rangle$. By Proposition \ref{odd_modSymb_construction}, we get an isomorphism $\phi : I\cdot H_+/I^2\cdot H_+ \xrightarrow{\sim} J^{(p)}\cdot \mathcal{K}^{(p)}/(J^{(p)})^2\cdot \mathcal{K}^{(p)} = J\cdot \mathcal{K}/J^2\cdot \mathcal{K}$ (the last identification comes from the norm, \cf Proposition \ref{Prop_K_theory} (\ref{Prop_K_theory_norm})). It remains to make this isomorphism explicit. We need to find an explicit lift in $H^{(p)}_+$ of an element in $I\cdot H_+$. By (\ref{Mazur_I/I^2}), the $\Zp$-module $I\cdot H_+$ is generated by elements of the form $\xi([xy]-[x]-[y])$ for $x$, $y$ $\in (\Z/N\Z)^{\times}$. Obviously, $\xi([xy]-[x]-[y])$ is the image of $[x,y^{-1}]^*_+ -[x,1]^*_+ +[y^{-1},1]^*_+ \in H^{(p)}_+$.  Thus, we have $$\phi([xy]-[x]-[y]) = \langle 1-\zeta_N^x, 1-\zeta_N^{y^{-1}} \rangle -\langle 1-\zeta_N^{x}, 1-\zeta_N \rangle + \langle 1-\zeta_N^{y^{-1}}, 1-\zeta_N \rangle \text{ .}$$
We can rewrite the right-hand-side as follows:
\begin{align*}
\langle 1-\zeta_N^x, 1-\zeta_N^{y^{-1}} \rangle &-\langle 1-\zeta_N^{x}, 1-\zeta_N \rangle + \langle 1-\zeta_N^{y^{-1}}, 1-\zeta_N \rangle \\& = [y^{-1}]\cdot \langle 1-\zeta_N^{xy}, 1-\zeta_N \rangle-\langle 1-\zeta_N^{x}, 1-\zeta_N \rangle - [y^{-1}]\cdot \langle 1-\zeta_N^{y}, 1-\zeta_N \rangle
\\& =  \langle 1-\zeta_N^{xy}, 1-\zeta_N \rangle-\langle 1-\zeta_N^{x}, 1-\zeta_N \rangle - \langle 1-\zeta_N^{y}, 1-\zeta_N \rangle +\\& ([y^{-1}]-1)\cdot \langle 1-\zeta_N^{xy}, 1-\zeta_N \rangle - ([y^{-1}]-1)\cdot \langle 1-\zeta_N^{y}, 1-\zeta_N \rangle \text{ .}
\end{align*}
Using the residue symbol $\mathcal{K}^{(p)}/J^{(p)}\cdot \mathcal{K}^{(p)} \xrightarrow{\sim} (\Z/N\Z)^{\times} \otimes_{\Z} \Zp$, we immediately check that the we have
\begin{align*}
([y^{-1}]&-1)\cdot \langle 1-\zeta_N^{xy}, 1-\zeta_N \rangle - ([y^{-1}]-1)\cdot \langle 1-\zeta_N^{y}, 1-\zeta_N \rangle \\&= -\frac{1}{2} \cdot( ([xy]-1)\cdot \langle 1-\zeta_N^{xy}, 1-\zeta_N \rangle-([x]-1)\cdot \langle 1-\zeta_N^{x}, 1-\zeta_N \rangle-([y]-1)\cdot \langle 1-\zeta_N^{y}, 1-\zeta_N \rangle)
\end{align*}
in $J^{(p)} \cdot \mathcal{K}^{(p)}/(J^{(p)})^2\cdot \mathcal{K}^{(p)}$. This proves that the map defined in Theorem \ref{thm_conj} coincides with $\phi$. This concludes the proof of Theorem \ref{thm_conj}.
\end{proof}

\section{The Eisenstein quotient conjecture}\label{section_Eisenstein_quotient}

The analogue of Conjecture \ref{Sharifi_conjecture} has been proved by Fukaya--Kato when the level of the modular curve is $Np^r$ ($r>0$) \cite[Theorem 5.2.3]{Fukaya_Kato}. In what follows, we adapt their method to prove some results for $X_1(N)$. In particular, we prove Theorem \ref{thm_Sharifi_mod_norms}.

\textit{Caveat about Hecke operators}. Our various (co)homology groups are equipped with an action of the Hecke operators $T_n$ for $n\geq 1$, induced by Albanese functoriality. We denote by $T_n^*$ the dual of $T_n$, induced by Picard functoriality. Fukaya and Kato instead consider the dual Hecke operators. This comes from the fact that they work with cohomology groups associated to modular curves, while we work with modular symbols which are homology groups. 
We have the following isomorphisms via Poincar\'e duality
\[
H^1(X_{1}(N)(\C),\Zp)(1)\cong H_1(X_{1}(N)(\C),\Zp),
\]
\[H^1(Y_{1}(N)(\C),\Zp)(1)\cong H_1(X_{1}(N)(\C),\textnormal{cusps},\Zp).
\]
These isomorphisms are not Hecke compatible, they transfer from the $T_{n}^{*}$- action to the $T_{n}$-action. For the details, one can see \cite[Proposition 3.5]{Sha1}. More generally, we will say that a map between two Hecke modules is \textit{anti-Hecke equivariant} if the action of $T_n$ on the left corresponds to the action of $T_n^*$ on the right.

\subsection{Eisenstein quotient conjecture and the \texorpdfstring{$\infty$}{Lg}-map}
In this section, we explain the relationship between Sharifi's conjecture and the ``$\infty$-map'' (which is intuitively a specialization at the cusp $\infty$).

\begin{defn}
	Let $\mathcal{Y}_{1}(N)$ be the $\Z[1/N]$-scheme that represents the functor taking a $\Z[1/N]$-scheme $S$ to the set of pairs $(E,\alpha)$, where $E$ is a elliptic curve over $S$ and $\alpha$ is a closed immersion $\Z/N\Z\rightarrow E$ of $S$-group schemes. Let $Y_1(N):=\mathcal{Y}_{1}(N)\otimes{\Q}$. 
\end{defn}

\begin{defn}
	Let $\rho': \tilde{H}_{\Gamma_1(N)}\rightarrow H_{\textnormal{\'et}}^2(\mathcal{Y}_1(N)\otimes\Z[1/p],\Zp(2))$ be the (a priori not well-defined) map such that
	\[\rho'([u,v]^*)=g_{0,\frac{u}{N}}\cup g_{0,\frac{v}{N}}.
	\]
	Here $g_{0,\frac{u}{N}}, g_{0,\frac{v}{N}}$ are Siegel units (Definition \ref{Def 5.1}). We call the elements $g_{0,\frac{u}{N}}\cup g_{0,\frac{v}{N}}$ Beilinson--Kato elements.
\end{defn}

\begin{conj}\label{Eisensein_rho'}
	The map $\rho'$ is well-defined and anti- Hecke-equivariant.
\end{conj}

\begin{rem}
	 Fran\c cois Brunault and Alexander Goncharov studied an analogue $\rho''$ of $\rho'$ for the principal congruence subgroup $\Gamma(N)$. Brunault proved that $\rho'' \otimes \Qp$ is well-defined \cite[Theorem 1.4]{BR}, and the proof of \cite[Theorem 4.2]{BR} shows that $\rho'\otimes \Qp$ is well-defined. Goncharov proved that $\rho''$ is well-defined \cite[Theorem 2.17]{GA}. In this paper, we use different method to study this map and do not rely on the results of Brunault or Goncharov.
\end{rem}

The map $\tilde{\varpi}$ of Proposition \ref{Prop_existence_varpi} extends to a map $\tilde{H}_{\Gamma_1(N)} \rightarrow \mathcal{K}$ sending $[u,v]^*$ to $\langle 1-\zeta_N^u, 1-\zeta_N^v \rangle$.  For simplicity of notation, we still denote this map by $\tilde{\varpi}$. Fukaya and Kato defined a map $\infty(0,1):H_{\textnormal{\'et}}^2(\mathcal{Y}_1(N)\otimes\Z[1/p],\Zp(2))\rightarrow H^2(\Z[\zeta_{N},\tfrac{1}{Np}],\Zp(2))$ \cite[Section 5.1]{Fukaya_Kato} and proved the following theorem.
\begin{thm}[Fukaya--Kato]\label{Fukayakato1}
	We have 
	\[
	\tilde{\varpi}([u,v]^{*})=\infty(0,1)(g_{0,\frac{u}{N}}\cup g_{0,\frac{v}{N}})
	\]
	Moreover, restricted to the image of the Beilinson--Kato elements, the map $\infty(0,1)$ is annihilated by the Hecke operators $T_{\ell}^*-\ell\cdot \langle \ell^{-1} \rangle -1$ for primes $\ell \neq N$ and by $T_{N}^*-1$.
\end{thm}

\begin{proof}
	For the proof, see \cite[Proposition 5.1.5, Proposition 5.1.7, Theorem 5.1.9]{Fukaya_Kato}.
\end{proof}

\begin{corr}\label{Cor2}
	If Conjecture \ref{Eisensein_rho'} is true, then Conjecture \ref{Sharifi_conjecture} is true. 
\end{corr}

\begin{defn}
	Let $\rho$ be the following map:
	\[ \rho:\tilde{H}_{\Gamma_1(N)}\xrightarrow{\rho'} H_{\textnormal{\'et}}^2(\mathcal{Y}_{1}(N)\otimes\Z[1/p],\Zp(2))\xrightarrow{g} H^1(\Z[\tfrac{1}{Np}],H_{\textnormal{\'et}}^1(Y_{1}(N))(2)).\]
	Here the map $g$ comes from the following exact sequence, which is derived from the Hochschild--Serre spectral sequence:
	\begin{equation*}
	0\rightarrow H^2(\Z[\tfrac{1}{Np}],\Zp(2))\rightarrow H_{\textnormal{\'et}}^{2}(\mathcal{Y}_{1}(N)\otimes\Z[1/p],\Zp(2))\overset{g}{\rightarrow}  H^1(\Z[\tfrac{1}{Np}],H_{\text{\'et}}^1(Y_{1}(N))(2)))\rightarrow 0.
	\end{equation*}	
\end{defn}

In the following paragraphs, following the method in \cite{Fukaya_Kato}, we prove the following result.
\begin{thm}\label{regulator_hecke_equivariant}
	The map $\rho$
	is well-defined and anti-Hecke-equivariant.
\end{thm}

\begin{rem}
	Based on the above discussion, we have the following commutative diagram:
	\begin{equation*}
	\begin{tikzcd}
	& &H^2(\Z[\zeta_{N}, \tfrac{1}{Np}],\Zp(2))/H^2(\Z[\tfrac{1}{Np}],\Zp(2))\\
	\tilde{H}_{\Gamma_1(N)}\arrow[urr, shift left=1.5ex,"\tilde{\varpi}"]\arrow[r,"\rho'"]& H_{\textnormal{\'et}}^2(\mathcal{Y}_{1}(N)\otimes\Z[1/p],\Zp(2)) \arrow[ur,"\infty"] \arrow[r,"g"] & H^1(\Z[\tfrac{1}{Np}], H_{\textnormal{\'et}}^1(Y_{1}(N))(2))\arrow[u] 
	\end{tikzcd}
	\end{equation*}
By Theorems \ref{Fukayakato1} and  \ref{regulator_hecke_equivariant}, the map 
\[\tilde{H}_{\Gamma_1(N)}\xrightarrow{\tilde{\varpi}}H^2(\Z[\zeta_{N}, \tfrac{1}{Np}],\Zp(2))\rightarrow H^2(\Z[\zeta_{N}, \tfrac{1}{Np}],\Zp(2))/H^2(\Z[\tfrac{1}{Np}],\Zp(2))\]
is annihilated by the Hecke operators $T_{\ell}-\ell\cdot \langle \ell \rangle -1$ for primes $\ell \neq N$ and by $T_{N}-1$.	
\end{rem}

\subsection{Beilinson--Kato elements}
Let $M_1,M_2\in\Z_{+}, M_1+M_2\geq 5$. Let $\mathcal{Y}(M_1,M_2)$ be the $\Z[\tfrac{1}{M_1M_2}]$-scheme that represents  the functor of triples $(E,e_{1},e_{2})$, where $e_{1}$ has order $M_1$, $e_{2}$ has order $M_2$, and  $\Z/M_{1}\Z\times\Z/M_{2}\Z\rightarrow E: (a,b)\mapsto ae_{1}+be_{2}$ is injective.

Following \cite[\S 1]{Kato_p-adic}, we recall the definition of Siegel units. 
\begin{defn}\label{Def 5.1}
For $(\alpha,\beta)=(\frac{a}{M_1},\frac{b}{M_2})\in (\frac{\frac{1}{M_1}\Z}{\Z} \times\frac{\frac{1}{M_2}\Z}{\Z})\backslash (0,0)$ and for an integer $c>1$ prime to $6M_1M_2$, the Siegel unit ${_{c}}g_{\alpha,\beta}$ is an element of $\O(\mathcal{Y}(M_1,M_2))^{\times}$. It is characterized by its $q$-expansion as follows. For $t\in \C$, let 
\[_{c}\theta(t)=q^{\tfrac{c^2-1}{12}}(-t)^{\tfrac{c-c^2}{2}}\prod_{n\geq 0}(1-q^nt)^{c^2}\cdot \prod_{n\geq 1}(1-q^nt^{-1})^{c^2}\cdot\prod_{n\geq 0}(1-q^nt^{c})^{-1}\cdot \prod_{n\geq 1}(1-q^nt^{-c})^{-1}\in\C[[q]][q^{-1}]^{\times}.\]
The Siegel unit ${_{c}}g_{\alpha,\beta}$ has the $q$-expansion
\[{_{c}}g_{\alpha,\beta}=_{c}\theta(q^{\tfrac{a}{M_1}}\zeta_{M_2}^b) \in \Z[\zeta_{M_2}, \tfrac{1}{M_1M_2}][[q^{\tfrac{1}{M_1}}]][q^{-1}]^{\times}.\]
Taking $c$ such that $c\equiv 1\bmod M_1$ and $c\equiv 1\bmod M_2$ let
\[g_{\alpha,\beta}={_{c}}g_{\alpha,\beta}\otimes(c^2-1)^{-1}\in \O(\mathcal{Y}(M_1,M_2))^{\times}\otimes\Q.\]
\end{defn}

In \cite{Fukaya_Kato}, the following objects are defined.
\begin{defn}[Fukaya--Kato]\label{Def5}
Let $R=\begin{psmallmatrix}
s&u\\t&v
\end{psmallmatrix}\in M_{2}(\Z)$ such that $(s,u)\neq(0,0)$ and $(t,v)\neq (0,0)$. 
	Let $c,d \in \Z$ coprime to $6M_1M_2$.  We define
\[_{c,d}z_{M_1,M_2}(R)={}_{c}g_{\frac{s}{M_1},\frac{u}{M_2}}\cup {}_{d}g_{\frac{t}{M_1},\frac{v}{M_{2}}}\in K_{2}(\mathcal{Y}(M_1,M_2)).
\]

Let $z_{M_1,M_2}(R)=g_{\frac{s}{M_1},\frac{u}{M_2}}\cup g_{\frac{t}{M_1},\frac{v}{M_2}}\in K_{2}(\mathcal{Y}(M_1,M_2))\otimes\Q$.
\end{defn}

From now on, assume that $M\geq 4$ and $m\geq 1$.
\begin{defn}[Fukaya--Kato]\label{Def 5.9}
Let $u,v\in\Z/M\Z$ such that $gcd(u,v,M)=1$. Take lifts $u',v'
\in\Z$ of $u$ and $v$ and integers $s,t$ such that $sv'-tu'=1$.  Let $_{c,d}z_{1,M,m}(u,v)$ be the image of $_{c,d}z_{m,Mm}\begin{psmallmatrix}
s&u'\\t&v'
\end{psmallmatrix}$ under the norm map $K_{2}(\mathcal{Y}(m,Mm))\rightarrow K_{2}(\mathcal{Y}_{1}(M)\otimes\Z[\frac{1}{Mm},\zeta_{m}])$. Here, the map from $\mathcal{Y}(m,Mm)$ to $\mathcal{Y}_{1}(M)\otimes\Z[\frac{1}{Mm},\zeta_{m}]$ is defined as follows:
\[(E,e_1,e_{2})\mapsto (E,me_{2},[e_1,Me_{2}]),\]
where $E$ is an elliptic curve and $e_{1}\in E[m], e_{2}\in E[mM]$, and $[\,\,,\,]$ is the Weil pairing. 
\end{defn}

The following result is derived from \cite[Section 2]{Fukaya_Kato}.
\iffalse
(we denote by a star the dual Hecke operator, induced by Picard functoriality).\fi

\begin{prop}\label{Prop8}
Let $L\geq 1$, and let $m,M$ be as in the beginning of this section. Then the norm map $K_2(Y_{1}(M)\otimes\Z[\frac{1}{mL},\zeta_{mL}])\rightarrow K_2(Y_{1}(M)\otimes\Z[\frac{1}{mL},\zeta_{m}])$ sends $_{c,d}z_{1,M,mL}(u,v)$  to
\[\prod_{\ell\in C'}(1-\sigma_{\ell}^{-1}\otimes T_{\ell}^{*}+\sigma_{\ell}^{-2}\otimes \langle \ell \rangle^{-1} \ell)\prod_{\ell\in C}(1-\sigma_{\ell}^{-1}\otimes T_{\ell}^{*} )_{c,d}z_{1,M,m}(u,v),\]
where $C'$ denotes the set of all prime numbers which divide $L$ but do not divide $mM$, $C$ denotes the set of all primes which divide both $L$ and $M$ but do not divide $m$, and $\sigma_{\ell}\in\Gal(\Q(\zeta_{m})/\Q)$ is such that $\sigma_{\ell}(\zeta_{m})=\zeta_{m}^{\ell}$.
\end{prop}
\begin{proof}
This follows directly from \cite[Proposition 2.2.2]{Fukaya_Kato}.
\end{proof}

\begin{notation}\label{Def6}
Let $_{c,d}[u,v]^*\in\Z[\Gal(\Q(\zeta_{m})/\Q)]\otimes H^1(Y_{1}(M)(\C),\Z)$
be
$$_{c,d}[u,v]^*=c^2d^2\otimes[u,v]^*-c^2\sigma_{d}\otimes[u,dv]^*-d^2\sigma_{c}\otimes[cu,v]^*+\sigma_{cd}\otimes[cu,dv]^*,$$
where $c,d$ are prime to $6Mm$.
\end{notation}

\subsection{The map \texorpdfstring{$z_{1,N,p^{\infty}}$}{Lg}}
In this section, we focus on the modular curve $Y_{1}(N)$. 
\iffalse
{\color{blue}{The notation $\Lambda$ has already been used (\cf introduction). I changed it to $\boldsymbol{\Lambda}$. Is it OK?}}
\fi
Let
\[\boldsymbol{\Lambda} :=\Zp[[\Zp^{\times}]]\cong\Zp[[\Gal(\Q(\zeta_{p^{\infty}})/\Q)]].\]

\begin{notation}
We have the following composition of maps
\[K_{2}(\mathcal{Y}_{1}(N)\otimes\Z[\zeta_{p^n}, \frac{1}{Np}])\xrightarrow{\textnormal{Chern class}} H_{\textnormal{\'et}}^{2}(\mathcal{Y}_{1}(N) \otimes\Z[\zeta_{p^n}, \frac{1}{Np}] ,\Zp(2))\rightarrow H^1(\Z[\zeta_{p^n}, \tfrac{1}{Np}],H_{\textnormal{\'et}}^1(Y_{1}(N))(2)).\]
To simplify the notation, we also use $_{c,d}z_{1,N,p^n}(u,v)$ to denote the image of $_{c,d}z_{1,N,p^n}(u,v)$ in $H^1(\Z[\zeta_{p^n}, \tfrac{1}{Np}],H_{\textnormal{\'et}}^1(Y_{1}(N))(2))$.
\end{notation}

\begin{rem}
Since $p\nmid N$, $z_{1,N,1}(u,v)=g_{0,\frac{u}{N}}\cup g_{0,\frac{v}{N}}\in H^1(\Z[\tfrac{1}{Np}],H_{\textnormal{\'et}}^1(Y_{1}(N))(2))$. 
\end{rem}

Fukaya and Kato proved the following theorem.
\begin{thm}[Fukaya--Kato]\label{Fukayakato2}
Let $Q(\boldsymbol{\Lambda})$ be the total quotient ring of $\boldsymbol{\Lambda}$. There exists a unique Hecke-equivariant $\boldsymbol{\Lambda}$-module homomorphism
$$z_{1,N,p^{\infty}}:\boldsymbol{\Lambda}\otimes_{\Z[\pm 1]} H^1(Y_{1}(N)(\C),\Z)(1)\rightarrow \varprojlim_{n} H^1(\Z[\zeta_{p^n}, \tfrac{1}{Np}], H^1_{\text{\'et}}(Y_{1}(N))(2))\otimes Q(\boldsymbol{\Lambda})$$
satisfying the following conditions. Here $-1$ in $\{\pm 1\}$ acts on $\boldsymbol{\Lambda}$ as $\sigma_{-1}$ and acts on $H^1(Y_1(N)(\C),\Z)$ by the complex
conjugation on $Y_1(N)(\C)$.
\begin{enumerate}
    %\item For any $\gamma\in\Lambda\otimes_{\Z[\pm 1]} H^1(Y_{1}(M)(\C),\Z)(1)$ and for any $c,d$ such that $(cd,6Np)=1$ and $c\equiv d\equiv 1\mod N$, $(\sigma_{c}-c)(\sigma_{d}-d)z_{1,N,p^{\infty}}(\gamma)$ belongs to the image of $\widetilde{\mathfrak{T}}$ in $\widetilde{\mathfrak{T}}\otimes Q(\Lambda)$.
    \item $z_{1,N,p^{\infty}}$ sends $_{c,d}[u,v]^{*}$ to $(-_{c,d}z_{1,N,p^n}(u,v))_{n\geq 1}$, and
    \item the image of $1\otimes[u,v]^{*}$ in $H^1(\Z[\zeta_{p^n}, \tfrac{1}{Np}], H_{\text{\'et}}^1(Y_{1}(N))(2))\otimes\Qp$ is $-z_{1,N,p^n}(u,v)$.
\end{enumerate}
\end{thm}

\begin{rem}\label{Rem9}
Theorem \ref{Fukayakato2} follows from \cite[Theorem 3.1.5, Proposition 3.1.8]{Fukaya_Kato}. In our setting, we do not assume that $p$ divides the level but the same proof still works.
\end{rem}

So in the first layer of this tower, we have a Hecke-equivariant map:
\[\Zp[\Gal(\Q(\zeta_{p})/\Q)]\otimes_{\Z[\pm 1]} H^1(Y_{1}(N)(\C),\Z)(1)\rightarrow H^1(\Z[\zeta_{p}, \tfrac{1}{Np}], H_{\text{\'et}}^1(Y_{1}(N))(2))\otimes\Qp,\]
which induces a map
\[Z:H^1(Y_{1}(N)(\C),\Z)(1)\rightarrow H^1(\Z[\zeta_{p}, \tfrac{1}{Np}], H_{\text{\'et}}^1(Y_{1}(N))(2))\otimes\Qp\]
sending $[u,v]^{*}$ to $-z_{1,N,p}(u,v)$.

\begin{prop}\label{Prop9}
The norm map 
$$H^1(\Z[\zeta_{p}, \tfrac{1}{Np}], H_{\textnormal{\'et}}^1(Y_{1}(N))(2))\otimes\Qp\rightarrow H^1(\Z[\tfrac{1}{Np}], H_{\textnormal{\'et}}^1(Y_{1}(N))(2))\otimes\Qp,$$
takes $z_{1,N,p}(u,v)$ to $(1-T_p^{*}+p\langle p\rangle ^{-1})z_{1,N,1}(u,v)$.
\end{prop}

\begin{proof}
It follows immediately from Proposition \ref{Prop8}.
\end{proof}

\begin{rem}\label{Rem10}
We thus have an anti-Hecke-equivariant map
\[Z':\tilde{H}_{\Gamma_1(N)}\rightarrow H^1(\Z[\tfrac{1}{Np}], H_{\text{\'et}}^1(Y_{1}(N))(2))\otimes \Qp\]
which sends $[u,v]^{*}$ to $(1-T_p^{*}+p\langle p\rangle ^{-1})z_{1,N,1}(u,v)$.
\end{rem}

\begin{lem}\label{Prop10}
The operator 
\[\eta_{p}:= 1-T_p^{*}+p\langle p\rangle ^{-1} \in \End_{\Qp}\big( H^1(\Z[\tfrac{1}{Np}],H_{\textnormal{\'et}}^1(X_{1}(N))(2)\otimes\Qp))\big)\] is injective.
\end{lem}

\begin{proof}
For simplicity of notation, let $V=H_{\text{\'et}}^1(X_{1}(N))(1)\otimes_{\Zp}\Qp$. We have an isomorphism $V\cong\bigoplus_{f_{i}}T_{p}(A_{f_{i}})\otimes_{\Zp} \Qp$, where the $f_{i}$'s are weight $2$ newforms of level $\Gamma_{1}(N)$ and $T_p(A_{f_i})$ is the associated $p$-adic Tate-module. Note that this isomorphism is anti-Hecke equivariant, so $\eta_{p}$ acts on $T_{p}(A_{f_{i}})\otimes\Qp$ by multiplication by $1-a_{p}(f)+p\chi_{i}(p)$, where $\chi_{i}$ is the character of $f_{i}$. 

We have $1-a_{p}(f)+p\chi_{i}(p)\neq 0$ since the roots of the polynomial 
 $x^2-a_p(f)x+p\chi_{i}(p)$ have absolute value $p^{1/2}$ (\cf \cite[\S 14.10.5]{Kato_p-adic}).
Thus, $\eta_{p}$ acts injectively on $V(1)$. Since $V(1)$ is a finite dimensional vector space, $\eta_{p}$ is an isomorphism of $V(1)$. We  conclude that $\eta_{p}$ acts injectively on $H^1(\Z[\tfrac{1}{Np}], V(1))$.
\end{proof}

\begin{prop}\label{Prop11}
We have an anti-Hecke-equivariant map
\[\rho: \tilde{H}_{\Gamma_1(N)}\rightarrow H^1(\Z[\tfrac{1}{Np}], H_{\textnormal{\'et}}^1(Y_{1}(N))(2))\otimes \Qp\]
which sends $[u,v]^{*}$ to $z_{1,N,1}(u,v)$.
\end{prop}

\begin{proof}
We have the following exact sequence 
\[0\rightarrow H_{\textnormal{\'et}}^1(X_{1}(N))(2)\otimes\Qp\rightarrow H_{\textnormal{\'et}}^1(Y_{1}(N))(2)\otimes\Qp\rightarrow C(1)\rightarrow 0,\]
where $C=\Qp[\mathcal{C}_{\Gamma_1(N)}]^0$. Note that $H^0(\Z[\tfrac{1}{Np}], C(1))=0$ since $\Gal(\overline{\Q}/\Q(\zeta_{N}))$ acts trivially on $C$. We get the following exact sequence in Galois cohomology
\begin{align}\label{Boundary_computation}
&0 \rightarrow H^1(\Z[\tfrac{1}{Np}],H_{\textnormal{\'et}}^1(X_{1}(N))(2)\otimes\Qp)\rightarrow  H^1(\Z[\tfrac{1}{Np}],H_{\textnormal{\'et}}^1(Y_{1}(N))(2)\otimes\Qp)\\
&\rightarrow H^1(\Z[\tfrac{1}{Np}], C(1))\nonumber.
\end{align}
The map $H^1(\Z[\tfrac{1}{Np}],H_{\textnormal{\'et}}^1(Y_{1}(N))(2)\otimes\Qp) \rightarrow H^1(\Z[\tfrac{1}{Np}], C(1))$ is denoted by $t$. Let $\gamma\in \tilde{H}_{\Gamma_1(N)}$. Suppose we have two different expressions for $\gamma$:
\begin{equation*}
\gamma=\sum_{i}a_{i}[u_i,v_i]^{*}=\sum_{j}b_{j}[u_{j},v_{j}]^{*},
\end{equation*}
where $a_{i}, b_{j}$ are elements in the Hecke algebra.
Let $A=\sum_{i}a_{i}z_{1,N,1}(u_{i},v_{i})$ and $B=\sum_{j}b_{j}z_{1,N,1}(u_{j},v_{j})$. By the boundary computation in \cite[Theorem 3.3.9 (iii)]{Fukaya_Kato}, we have $t(A-B)=0$. By (\ref{Boundary_computation}), we have $A-B\in H^1(\Z[\tfrac{1}{Np}],H_{\textnormal{\'et}}^1(X_{1}(N))(2)\otimes\Qp)$. By Remark \ref{Rem10}, we know that $\eta_{p}(A-B)=0$. From Lemma \ref{Prop10}, we know that $\eta_{p}$ is injective on $H^1(\Z[\tfrac{1}{Np}],H_{\textnormal{\'et}}^1(X_{1}(N))(2)\otimes\Qp)$. So, we can conclude that $A=B$.
\end{proof}

\begin{lem}\label{Prop12}
The group $H^1(\Z[\tfrac{1}{Np}],H_{\textnormal{\'et}}^1(Y_{1}(N))(2))$ is $\Zp$-torsion-free.
\end{lem}

\begin{proof}
Consider the exact sequence
\[0\rightarrow H_{\text{\'et}}^1(X_{1}(N))\xrightarrow{i} H_{\text{\'et}}^1(Y_{1}(N))\rightarrow C\rightarrow 0,\]
where $C$ is the cokernel of $i$. In order to prove that $H^1(\Z[\tfrac{1}{Np}],H_{\text{\'et}}^1(Y_{1}(N))(2))$ is $\Zp$-torsion-free, it suffices to prove that both $H^1(\Z[\tfrac{1}{Np}],H_{\text{\'et}}^1(X_{1}(N))(2))$ and $H^1(\Z[\tfrac{1}{Np}],C(2))$ are $\Zp$-torsion-free. From the sequence
\[0\rightarrow \Zp(2)\xrightarrow{p} \Zp(2)\rightarrow \mu_{p}^{\otimes 2}\rightarrow 0,\]
it suffices to prove that 
\begin{enumerate}
    \item\label{item_H_0_trivial_i} $H^0(\Z[\tfrac{1}{Np}],H_{\text{\'et}}^1(X_{1}(N))(2)\otimes \Fp)=0$,
    \item\label{item_H_0_trivial_ii} $H^0(\Z[\tfrac{1}{Np}],C(2)\otimes \Fp)=0$.
\end{enumerate}
Note that the $\Gal(\overline{\Q}_p/\Qp)$-representation $H_{\text{\'et}}^1(X_{1}(N))(2)\otimes\Qp$ is crystalline. Then (\ref{item_H_0_trivial_i}) is deduced from the statement that $H^0(\Qp,H_{\text{\'et}}^1(X_{1}(N))(2)\otimes \Fp)=0$. For the details, see \cite[Section 3.2]{FM}. For (\ref{item_H_0_trivial_ii}), it is true because $\Gal(\overline{\Q}/\Q(\zeta_{N}))$ acts trivially on $C(1)$.
\end{proof}
\begin{corr}\label{Cor3}
The image of the map $\rho$ is contained in $H^1(\Z[\tfrac{1}{Np}], H_{\textnormal{\'et}}^1(Y_{1}(N))(2))$, and the following compositum is annihilated by the Hecke operators $T_{\ell}-\ell\cdot \langle \ell \rangle -1$ for primes $\ell \neq N$ and by $T_{N}-1$:
\[\tilde{\varpi}':\tilde{H}_{\Gamma_1(N)}\xrightarrow{\tilde{\varpi}}H^2(\Z[\zeta_{N}, \tfrac{1}{Np}],\Zp(2))\rightarrow H^2(\Z[\zeta_{N}, \tfrac{1}{Np}],\Zp(2))/H^2(\Z[\tfrac{1}{Np}],\Zp(2)).\]
\end{corr}
\begin{lem}
The image of $H^2(\Z[\tfrac{1}{Np}],\Zp(2))$ in $H^2(\Z[\zeta_{N}, \tfrac{1}{Np}],\Zp(2))$ is $\nu.H^2(\Z[\zeta_{N}, \tfrac{1}{Np}],\Zp(2))$ (recall from the introduction that $\nu = \sum_{g \in \Gal(\Q(\zeta_N)/\Q)} [g]$).
\end{lem}

\begin{proof}
By \cite[Propositions 8.3.18 and 3.3.11]{Neukirch_Galois}, the corestriction map  $H^2(\Z[\zeta_{N}, \tfrac{1}{Np}],\Zp(2)) \xrightarrow{\textnormal{cor}}H^2(\Z[\tfrac{1}{Np}],\Zp(2))$ is surjective. Thus, the image of $H^2(\Z[\tfrac{1}{Np}],\Zp(2))$ in $H^2(\Z[\zeta_{N}, \tfrac{1}{Np}],\Zp(2))$ is the image of the compositum
\[H^2(\Z[\zeta_{N}, \tfrac{1}{Np}],\Zp(2)) \xrightarrow{\textnormal{cor}}H^2(\Z[\tfrac{1}{Np}],\Zp(2))\xrightarrow{\textnormal{res}}H^2(\Z[\zeta_{N}, \tfrac{1}{Np}],\Zp(2))\]
where $\text{res}$ is the restriction. 
By \cite[Corollary 1.5.7]{Neukirch_Galois}, the endomorphism $\text{res} \circ \text{cor}$ of $H^2(\Z[\zeta_{N}, \tfrac{1}{Np}],\Zp(2))$ is the multiplication by $\nu$. This concludes the proof of the lemma.
\end{proof}

To summarize, we have proved in this section the following result, which implies Theorem \ref{thm_Sharifi_mod_norms} (by restriction to $(\tilde{H}_{\Gamma_1(N)})_+$).
	
\begin{thm}
The following compositum is annihilated by the Hecke operators $T_{\ell}-\ell\cdot \langle \ell \rangle -1$ for primes $\ell \neq N$ and by $T_{N}-1$.
\[\tilde{H}_{\Gamma_1(N)}\xrightarrow{\tilde{\varpi}} H^2(\Z[\zeta_{N}, \tfrac{1}{Np}],\Zp(2))\rightarrow H^2(\Z[\zeta_{N}, \tfrac{1}{Np}],\Zp(2))/\nu.H^2(\Z[\zeta_{N}, \tfrac{1}{Np}],\Zp(2)).\]
\end{thm}

\section{Proof of Theorems \ref{thm_conj_2} and \ref{I/I^2_modulo_p}}

In this last section, we explain how to deduce Theorems \ref{thm_conj_2} and \ref{I/I^2_modulo_p} from Theorem \ref{thm_Sharifi_mod_norms}. This is similar to the proof of Theorem \ref{thm_conj}. 

Fix $s$ such that $1 \leq s \leq t$. Let $\overline{\mathcal{K}} = \mathcal{K} \otimes_{\Zp} \Z/p^s\Z$ and $\overline{H}_+=H_+\otimes_{\Zp} \Z/p^s\Z$. Let $\log : (\Z/N\Z)^{\times} \rightarrow \Z/p^s\Z$ be a surjective group homomorphism. Let $1 \leq v \leq s$ be the $p$-adic valuation of the \textit{Merel number} $\log(\prod_{k=1}^{\frac{N-1}{2}}k^k)=\sum_{k=1}^{\frac{N-1}{2}} k\cdot \log(k) \in \Z/p^s\Z$. Recall the following result, essentially due to Merel.

\begin{thm}[Merel]\label{Thm_structure_I_H_bar}
The group $I\cdot \overline{H}_+/I^2\cdot \overline{H}_+$ is isomorphic to $\Z/p^v\Z$. Thus, we have a canonical group isomorphism $$(I\cdot H_+/I^2\cdot H_+)\otimes_{\Zp} \Z/p^v\Z \xrightarrow{\sim} I\cdot \overline{H}_+/I^2\cdot \overline{H}_+ \text{ .}$$ 
\end{thm}
\begin{proof}
In the notation of \cite[Theorem 2.1]{Lecouturier_higher_Eisenstein}, this follows from the fact that $n(1)=v$, which is a direct consequence of \cite[Th\'eor\`eme 1]{Merel_accouplement} by \cite[Corollary 2.5 and Theorem 6.4]{Lecouturier_higher_Eisenstein}.
\end{proof}

On the $K$-theoretic side, we have the following result.
\begin{lem}\label{Lemma_structure_J_K_bar}
The group $J\cdot \overline{\mathcal{K}}/J^2\cdot \overline{\mathcal{K}}$ is isomorphic to $\Z/p^v\Z$. Thus, we have a canonical group isomorphism $$(J\cdot \mathcal{K}/J^2\cdot\mathcal{K})\otimes_{\Zp} \Z/p^v\Z \xrightarrow{\sim} J\cdot \overline{\mathcal{K}}/J^2\cdot \overline{\mathcal{K}} \text{ .}$$ 
\end{lem}
\begin{proof}
It suffices to prove the analogous statement where $\overline{\mathcal{K}}$ is replaced by $\overline{\mathcal{K}}^{(p)} := \mathcal{K}^{(p)} \otimes_{\Zp} \Z/p^s\Z$.

Recall that we have a isomorphism of $\Lambda^{(p)}$-modules $\mathcal{K}^{(p)}\simeq \Lambda^{(p)}/(\zeta^{(p)})$ by Proposition \ref{Prop_K_theory} (\ref{Prop_K_theory_Lambda}). Since the degree of $\zeta^{(p)}$ is $-\frac{N-1}{6N} \neq 0$, we have $J^{(p)} \cap (\zeta^{(p)}) = J^{(p)} \cdot (\zeta^{(p)})$. Thus, we have $$J^{(p)} \cdot \mathcal{K}^{(p)}/(J^{(p)})^2 \cdot \mathcal{K}^{(p)}\simeq J^{(p)}/((J^{(p)})^2+J^{(p)}\cdot (\zeta^{(p)})) \simeq \Lambda^{(p)}/(J^{(p)}+(\nu^{(p)})+(\zeta^{(p)}))\simeq \Z/p^t\Z$$
since the $p$-adic valuation of the degree of $\zeta^{(p)}$ and $\nu^{(p)}$ is $t$.

Since the degree of $\zeta^{(p)}$ is $0$ modulo $p^s$, we have $\overline{\zeta}^{(p)} \in \overline{J}^{(p)}$ where $\overline{\zeta}^{(p)}$ and $\overline{J}^{(p)}$ are the respective images of $\zeta^{(p)}$ and $J^{(p)}$ in $\overline{\Lambda}^{(p)}= \Lambda^{(p)} \otimes_{\Zp} \Z/p^s\Z$. Thus, we have:
$$J^{(p)}\cdot \overline{\mathcal{K}}^{(p)}/(J^{(p)})^2\cdot \overline{\mathcal{K}}^{(p)} \simeq \overline{J}^{(p)}/((\overline{J}^{(p)})^2+(\overline{\zeta}^{(p)})) \text{ .}$$
There is a group isomorphism $\overline{J}^{(p)}/(\overline{J}^{(p)})^2\xrightarrow{\sim} \Z/p^s\Z$ sending $\sum_{x \in (\Z/N\Z)^{\times}} \lambda_x \cdot [x]$ to $\sum_{x \in (\Z/N\Z)^{\times}} \lambda_x \cdot \log(x)$. The image of $\overline{\zeta}^{(p)}$ in $\Z/p^s\Z$ via this isomorphism is $\sum_{x \in (\Z/N\Z)^{\times}} \B_2(\frac{x}{N})\cdot \log(x)$. By \cite[Proposition 1.2]{Lecouturier_class_group}, we have 
$$\sum_{x \in (\Z/N\Z)^{\times}} \B_2(\frac{x}{N})\cdot \log(x) = -\frac{4}{3}\cdot \sum_{k=1}^{\frac{N-1}{2}} k \cdot \log(k) \text{ .}$$
This proves that $J^{(p)}\cdot \overline{\mathcal{K}}^{(p)}/(J^{(p)})^2\cdot \overline{\mathcal{K}}^{(p)}$ has order $p^v$, which concludes the proof of Lemma \ref{Lemma_structure_J_K_bar}.
\end{proof}

We first prove Theorem \ref{I/I^2_modulo_p}. 
\begin{proof}
As in section \ref{section_Sharifi_conj}, Theorem \ref{thm_Sharifi_mod_norms} shows that $\tilde{\varpi}^{(p)}$ induces a surjective group homomorphism
$$H^{(p)}_+/I_0\cdot H^{(p)}_+  \twoheadrightarrow J^{(p)}\cdot (\mathcal{K}^{(p)}/\nu^{(p)}\cdot \mathcal{K}^{(p)}) $$
and thus a surjective group homomorphism
$$I\cdot H_+/I^2\cdot H_+ = H^{(p)}_+/(I_0+J^{(p)})\cdot H^{(p)}_+  \twoheadrightarrow J^{(p)}\cdot (\mathcal{K}^{(p)}/\nu^{(p)}\cdot \mathcal{K}^{(p)})/(J^{(p)})^2\cdot (\mathcal{K}^{(p)}/\nu^{(p)}\cdot \mathcal{K}^{(p)}) \text{ .}$$
Since $s\leq t$, we have $\nu^{(p)}\cdot \overline{\mathcal{K}}^{(p)} \subset (J^{(p)})^2\cdot \overline{\mathcal{K}}^{(p)}$. We thus have a surjective map
$$I\cdot H_+/I^2\cdot H_+ \twoheadrightarrow J^{(p)}\cdot \overline{\mathcal{K}}^{(p)}/ (J^{(p)})^2\cdot \overline{\mathcal{K}}^{(p)} \text{ .}$$
Similarly as in Proposition \ref{Prop_K_theory} (\ref{Prop_K_theory_norm}), the norm map induces an isomorphism 
$$
J\cdot \overline{\mathcal{K}}/ J^2\cdot \overline{\mathcal{K}}  \xrightarrow{\sim} J^{(p)}\cdot \overline{\mathcal{K}}^{(p)}/ (J^{(p)})^2\cdot \overline{\mathcal{K}}^{(p)} \text{ .}
$$
We thus get a surjective map
$$\psi : I\cdot H_+/I^2\cdot H_+ \twoheadrightarrow J\cdot \overline{\mathcal{K}}/ J^2\cdot \overline{\mathcal{K}} \text{ .}$$
By Theorem \ref{Thm_structure_I_H_bar} and Lemma \ref{Lemma_structure_J_K_bar}, $\psi$ induces an isomorphism $\overline{\psi} : I\cdot \overline{H}_+/I^2\cdot \overline{H}_+ \xrightarrow{\sim} J\cdot \overline{\mathcal{K}}/J^2\cdot \overline{\mathcal{K}}$. The fact that $\overline{\psi}$ is given by the formula of Theorem \ref{I/I^2_modulo_p} is similar to the proof of Theorem \ref{thm_conj}.
\end{proof}

Under the assumption of Theorem \ref{thm_conj_2}, namely that $v=t$, we have the following commutative diagram whose vertical arrows are isomorphisms:
$$\xymatrix{
 I\cdot H_+/I^2\cdot H_+ \ar[r] \ar[d]^{\sim}  & J\cdot \mathcal{K}/J^2\cdot \mathcal{K} \ar[d]^{\sim}  \\
    I\cdot \overline{H}_+/I^2\cdot \overline{H}_+  \ar[r] & J\cdot \overline{\mathcal{K}}/J^2\cdot \overline{\mathcal{K}} }$$

Thus, Theorem \ref{thm_conj_2} is a consequence of Theorem \ref{I/I^2_modulo_p}.

\bibliography{biblio}
\bibliographystyle{plain}
\newpage
\end{document}